\documentclass[12pt]{amsart}
\usepackage[T1]{fontenc}
\usepackage[latin1]{inputenc}
\usepackage{typearea}
\usepackage{geometry}
\usepackage{ulem}
\usepackage{amsmath}
\usepackage{amssymb}
\usepackage{latexsym}
\usepackage{enumerate}
\usepackage{amsthm}
\usepackage[all]{xy}
\usepackage{hhline}
\usepackage{epsf} %fuer bildchen
\usepackage{cite}

\newtheorem{theorem}{{\sc Theorem}}[section]
\newtheorem{corollary}[theorem]{{\sc Corollary}}

\newtheorem{lemma}[theorem]{{\sc Lemma}}
\newtheorem{proposition}[theorem]{{\sc Proposition}}

\theoremstyle{remark}
\newtheorem{remark}[theorem]{{\sc Remark}}

\newtheorem*{acknowledgements}{\sc Acknowledgments}

\theoremstyle{definition}
\newtheorem{definition}[theorem]{\sc definition}
\newtheorem{example}[theorem]{\sc example}

\newcommand{\R}{\mathbb{R} }
\newcommand{\N}{\mathbb{N} }

\newcommand{\A}{\mathcal{A}}
\newcommand{\B}{\mathcal{B}}
\newcommand{\F}{\mathcal{F}}

\newcommand{\W}{\mathcal{W}}

\newcommand{\calL}{\mathcal{L}}

\newcommand{\al}{\alpha}

\newcommand{\Om}{\Omega}

\newcommand{\abquer}{\overline{(a,b)}}

\newcommand{\XB}{X^{(B)}}
\newcommand{\Btilde}{\tilde{B}}
\newcommand{\XBtilde}{X^{(\tilde{B})}}
\newcommand{\Xhat}{\hat{X}}
\newcommand{\XBm}{X^{(B,m)}}

\providecommand{\abs}[1]{\lvert #1\rvert}
\providecommand{\babs}[1]{\bigl\lvert #1\bigr\rvert}
\providecommand{\Babs}[1]{\Bigl\lvert #1\Bigr\rvert}

\providecommand{\fnorm}[1]{\lVert #1\rVert_\infty}
\DeclareMathOperator{\Var}{Var}

\DeclareMathOperator{\sign}{sign}
\DeclareMathOperator{\Lip}{Lip}

\renewcommand{\phi}{\varphi}
\renewcommand{\epsilon}{\varepsilon}
\newcommand{\eps}{\varepsilon}
\renewcommand{\rho}{\varrho}

\begin{document}
\title{Distributional transformations without orthogonality relations}
\author{Christian D\"obler}
\thanks{Technische Universit\"at M\"unchen, Fakult\"at f\"ur Mathematik, Lehrstuhl f\"ur Mathematische Physik, D-85740 M\"unchen, Germany. \\
christian.doebler@tum.de}
\begin{abstract}
Distributional transformations characterized by equations relating expectations of test functions weighted by a given biasing function on the original distribution to expectations of the test function's higher derivatives with respect to the transformed distribution play a great role in Stein's method and were, in great generality, first considered by Goldstein and Reinert \cite{GolRei05b}. We prove two abstract existence and uniqueness results for such distributional transformations, generalizing their $X-P$ bias transformation.
On the one hand, we show how one can abandon previously necessary orthogonality relations by subtracting an explicitly known polynomial depending on the test function from the test function itself. On the other hand, we prove that for a given nonnegative integer $m$ it is possible to obtain the expectation 
of the $m$-th derivative of the test function with respect to the transformed distribution in the defining equation, even though the biasing function may have $k<m$ sign changes, if these two numbers have the same parity.
We explain, how these results can be used to guarantee the existence of two different generalizations of the zero bias transformation by Goldstein and Reinert \cite{GolRei97}. Further applications include the derivation of Stein 
type characterizations without needing to solve any Stein equation and the presentation of a general framework of estimating the distance of the distribution of a given real random variable $X$ to that of a random variable 
$Z$, whose distribution is characterized by some $m$-th order linear differential operator. We also explain the fact that, in general, the biased distribution depends on the choice of the sign change points, if these are ambiguous. 
This new phenomenon does not appear in the framework from \cite{GolRei05b}.
\end{abstract}
\maketitle

\section{Introduction}
\label{intro}
Distributional transformations play a great role in Stein's method in connection with certain coupling constructions, which are often an essential tool for bounding the quantities arising from Stein's equation. Important and well studied examples are given by the well known size bias transformation (see e.g. \cite{GolRin96}, \cite{AGK} or \cite{ArrGol}) 
and the zero bias transformation, which was introduced in \cite{GolRei97} for mean zero random variables with finite and strictly positive variance. For an introduction to Stein's method for normal approximation we refer to the book \cite{CGS}, which includes an extensive discussion of the use of various coupling constructions in Stein's method. For a general introduction to Stein's method we refer to the book \cite{BarCh}.\\
Recall that for a nonnegative random variable $X$ with finite and positive mean $\mu$ one says that a random variable $X^s$ has the $X$\textit{-size biased distribution} if the identity $E[Xf(X)]=\mu E[f(X^s)]$ holds for all bounded and measurable functions $f$ on $[0,\infty)$. Existence of the $X$-size biased distribution is easily seen by just 
letting the distribution of $X^s$ have Radon-Nikodym derivative $x\mapsto x/\mu$ with respect to the distribution of $X$. In contrast, if $X$ is a given real-valued random variable with variance $\sigma^2\in(0,\infty)$ and if $E[X]=0$, then a random variable $X^*$ is said to have the $X$\textit{-zero biased distribution} if 
$E[Xf(X)]=\sigma^2E[f'(X^*)]$ for all Lipschitz continuous functions $f$ on $\R$. It was shown in \cite{GolRei97} that, for a given real-valued random variable $X$, the $X$-zero biased distribution exists uniquely if and only if $0<E[X^2]<\infty$ and $E[X]=0$ and that the distribution of $X^*$ is always absolutely continuous with respect to the Lebesgue 
measure.\\
In \cite{GolRei05b}, given a real-valued random variable $X$, an integer $m\geq0$  and a function $P$ on $\R$, Goldstein and Reinert addressed the general problem of when a random variable $X^{(P)}$ and a constant $\alpha$ exist such that 
\begin{equation}\label{GRid}
E\bigl[P(X)F(X)\bigr]=\alpha E\bigl[F^{(m)}\bigl(X^{(P)}\bigr)\bigr] 
\end{equation}
holds for a sufficiently large class of functions $F$ on $\R$. Their most general result in this direction, Theorem 2.1 of \cite{GolRei05b}, guarantees existence and uniqueness of the distribution for such a random variable $X^{(P)}$, if $P$ has exactly $m$ sign changes on $\R$ and if there exists an $\alpha>0$ such that the \textit{orthogonality relations} 
$E[X^kP(X)]=m! \alpha\delta_{k,m}$ hold for $k=0,1,\dotsc,m$. \\
The main purpose of the present paper is to generalize Theorem 2.1 of \cite{GolRei05b} in two respects: Firstly, in Theorem \ref{maintheo} below, we make sure that one can do without the orthogonality relations by replacing the term $F(X)$ on the left hand side of \eqref{GRid} by $F(X)-L_F(X)$, where $L_F$ is an explicit polynomial of degree 
at most $m-1$, which depends on $F$ and the sign change points of $P$. We further show that the distribution of $X^{(P)}$ is always absolutely continuous with respect to the Lebesgue measure, if $m\geq1$. Secondly, we consider the case that the number $k$ of sign changes of the function $P$ is strictly smaller than the order $m$ of the derivative we would like 
to have on the right hand side of \eqref{GRid}. Our general existence and uniqueness result, Theorem \ref{genmaintheo}, which is in fact a generalization of Theorem \ref{maintheo}, makes sure that the desired distributional transformation exists, if we additionally assume that $k$ and $m$ have the same parity.
%We also hint at how Theorem \ref{genmaintheo} 
%(or, in fact, Proposition \ref{mainprop}) can be applied to improve results from \cite{PiRen} on convergence rates of certain random sums to the Laplace distribution.
Although this paper is mainly intended to extend the abstract theory of distributional transformations, we present some worked out and some potential applications of this theory to first and higher order Stein operators 
in Section \ref{examples}. The paper is structured as follows: In Section \ref{s2} we state and prove our main abstract theorems, Theorem \ref{maintheo} and Theorem \ref{genmaintheo}. In Subsection \ref{s21}, we state 
Theorem \ref{maintheo} on $m$ sign changes and give a probabilistic proof, whereas in Subsection \ref{s22} we consider the more general case of $k\leq m$ sign changes. 
In Section \ref{examples} we illustrate how our abstract results can be used to prove Stein type characterizations of distributions without solving the Stein equation and prove two more abstract results, which guarantee 
the existence of certain distributional transformations corresponding to some higher order linear Stein operators. We show in the one- and two-dimensional cases, how these results can in principle be used to estimate the distance of a given distribution on $\R$ to one that is a fixed point of the distributional transformation, which is associated to the given linear operator. We also give a simple example which makes clear that 
the distributional transformations considered here, are in general sensitive to the specific choice of sign changes of the biasing function, if these are ambiguous. 
Finally, in Section \ref{anproof} we give an analytical proof of Theorem \ref{maintheo}, which invokes the Riesz representation theorem and which was our original proof of this result. 
Since we prefer to use the symbol $P$ for probabilty measures, from now on we denote the biasing function by $B$ instead.
 
\section{Main abstract results and discussion}\label{s2}
Let $B:\R\rightarrow\R$ be a measurable function, which will henceforth be called a \textit{biasing function}. We say $B$ has no sign changes, if it is either nonnegative or nonpositive on $\R$. For $m\in\{1,2,\dotsc\}$ we say that $B$ has $m$ sign changes occuring at the points $x_1<x_2<\ldots<x_m$, if the following holds: Letting  
$J_1:=(-\infty,x_1], J_2:=(x_1,x_2],\dotsc, J_m:=(x_{m-1},x_m]$ and $J_{m+1}:=(x_m,\infty)$ we either have for each integer $1\leq k\leq m+1$ and for all $x\in J_k$ that $(-1)^{m+1-k} B(x)\geq0$ or for each integer $1\leq k\leq m+1$ and for all $x\in J_k$ it holds that $(-1)^{m+1-k} B(x)\leq0$. Note that the first contingency is equivalent to 
$\prod_{k=1}^m(x-x_k)B(x)\geq0$ for all $x\in\R$, whereas the second case means the same as 
$\prod_{k=1}^m(x-x_k)B(x)\leq0$ on $\R$. As was already noted in \cite{GolRei05b}, the points where the sign changes occur may not be 
unique if there are non-trivial subintervals of $\R$, where $B$ is identically equal to zero. Note also that this definition is slightly more general than the one in \cite{GolRei05b} in that they would additionally demand the existence of an $x\in J_k$ with $B(x)\not=0$ for each $k=1,2,\dotsc, m+1$. This generalization also implies that a given $B$ may be 
considered to have both, $m$ and $m'$ sign changes, where $m\not=m'$, if there are non-trivial subintervals of $\R$, where $B$ is identically equal to zero. Throughout, for $m\in\{1,2,\dotsc\}$, we denote by $\F^m$ the class 
of all functions $F\in C^{m-1}(\R)$ such that $F^{(m-1)}$ is still Lipschitz-continuous. For $m=0$, we denote by $\F^0$ the class of all bounded and Borel-measurable functions on $\R$. Further, we adopt the standard conventions that empty sums are set equal to zero and empty products are set equal to one.

For a function $F$ on $\R$ and $m\geq0$ real numbers $x_1<x_2<\ldots<x_m$ we define the polynomial $L_F:=L_{F;x_1,\dotsc,x_m}$ of degree at most $m-1$ by $L_F:=0$ if $m=0$ and for $m\geq1$ we define $L_F$ to be the interpolation polynomial corresponding to the function $F$ and the nodes $x_1,\dotsc,x_m$, i.e.
\begin{equation}\label{LF}
 L_F(x):=\sum_{k=1}^{m} F(x_k)\prod_{\substack{j=1\\ j\not=k}}^m\frac{x-x_j}{x_k-x_j}\,.
\end{equation}

\subsection{Biasing functions with $m$ sign changes}\label{s21}
In this Subsection we give a proof of the following theorem, which is a generalization of Theorem 2.1 in \cite{GolRei05b}.

\begin{theorem}\label{maintheo}
 Let $m$ be a nonnegative integer and let $B:\R\rightarrow\R$ be a measurable biasing function having $m\geq0$ sign changes at the points $x_1<x_2<\ldots<x_m$. If $m=0$, suppose that $B$ is nonnegative on $\R$ and if $m\geq1$, suppose that $B$ is nonnegative on $J_{m+1}$. Assume further that $X$ is a real-valued random variable on some probability space $(\Om,\A,P)$ such that $E\abs{X^j B(X)}<\infty$ for $j=0,1,\dotsc,m$ and 
\begin{equation*}
 \alpha:=\frac{1}{m!}E\Bigl[B(X)(X-x_m)(X-x_{m-1})\cdot\ldots\cdot(X-x_1)\Bigr]\not=0\,.
\end{equation*}
Then, $\alpha$ is necessarily positive and there exists a unique distribution for a random variable $X^{(B)}$ such that for all $F\in\F^m$ we have 
\begin{align}\label{mainid}
 \alpha E\Bigl[F^{(m)}\bigl(X^{(B)}\bigr)\Bigr]=E\Bigl[B(X)\bigl(F(X)-L_F(X)\bigl)\Bigr]\,,
\end{align}
whith $L_F$ as defined in \eqref{LF}. Furthermore, if $m\geq1$, then the distribution of $X^{(B)}$ is absolutely continuous with respect to the Lebesgue measure.
\end{theorem}

\begin{remark}\label{remmt}
\begin{enumerate}[(a)]
\item If $X$ and $B$ additionally satisfy the orthogonality conditions\\
$E[X^j B(X)]=0$ for all $j=0,1,\dotsc,m-1$, then the distribution $\calL(\XB)$ of $\XB$ reduces to the $X-B$ biased distribution from \cite{GolRei05b} as is easily seen by writing the polynomial $L_F$ in terms 
of the monomials $1,X,\dotsc,X^{m-1}$. Also, in this case for the same reason we have $\alpha=(m!)^{-1}E[X^mB(X)]$. So it is justified to call the distribution of $\XB$ the \textit{generalized $X-B$ biased distribution}. 
\item Note that if, according to our definition of sign changes,  $B$ has both, $m$ and $m'$ sign changes for $m\not=m'$, then we see from \eqref{mainid} that these two points of view lead to different distributions for $\XB$. Also, if we may consider $B$ to have sign changes at $x_1<\ldots<x_m$ as well as at $y_1<\ldots<y_m$, then the 
resulting $\alpha$'s and, again, the distributions of $\XB$'s are different, in general, which is in contrast to the theory from \cite{GolRei05b}, where such ambiguities are ruled out by their orthogonality asumptions on $X$ with respect to $B$. Thus, one should actually denote the variable $\XB$ by $X^{(B;x_1,\dotsc,x_m)}$ to 
prevent these ambiguities. We illustrate this phenomenon for the case $m=1$ in Example \ref{ambi} below.
We will, however, not do so but rather assume that it is understood or mention how many sign changes at what exact points the function $B$ is supposed to have.    
\item For the existence part of Theorem \ref{maintheo} we give two different proofs: An analytical proof, which uses the Riesz representation theorem, and a probabilistic proof, which relies on an explicit construction of the random variable $\XB$. Remarkably, the same construction of $\XB$ as in \cite{GolRei05b} is still valid in this 
more general setting. However, we were not able to generalize the proof of Theorem 2.1 in \cite{GolRei05b} to a proof of our Theorem \ref{maintheo}. 
\item In the case $m=1$, one can easily show that the function $p$ given by 
\begin{equation*}
 p(t)=\frac{1}{\alpha}E\Bigl[B(X)\bigl(1_{\{x_1\leq t\leq X\}}-1_{\{X< t<x_1\}}\bigr)\Bigr]
\end{equation*}
is a probability density function on $\R$, whose associated distribution satisfies the requirements for the generalized $X-B$ biased distribution, thus yielding a direct proof of existence and absolute continuity in this case.
\item Note that if $F\in\F^m$, then one can easily show by induction on $k=0,1,\dotsc,m$ that there exist finite constants $c_k>0$ such that 
\begin{equation*}
 \abs{F^{(m-k)}(x)}\leq c_k\bigl(1+\abs{x}^{k}\bigr)
\end{equation*}
for each $x\in\R$. Hence, if $X$ satisfies the conditions from Theorem \ref{maintheo}, then\\
$E\bigl[B(X)F(X)\bigr]$ exists for each $F\in\F^m$.
\item The assumption $E\abs{X^j B(X)}<\infty$ for $j=0,1,\dotsc,m$ is easily seen to be equivalent to $E\abs{B(X)}<\infty$ and $E\abs{X^mB(X)}<\infty$.
\end{enumerate}
\end{remark}

\begin{proof}[Proof of uniqueness in Theorem \ref{maintheo}]
The argument for uniqueness is the same as in \cite{GolRei05b} and is only included for reasons of completeness. Let $\mu$ and $\nu$ both be probability measures on $(\R,\B(\R))$ such that random variables $V\sim\mu$ and $W\sim\nu$ satisfy the conditions on $\XB$, i.e.
\begin{align}\label{uni}
\alpha E\bigl[F^{(m)}(V)\bigr]=E\Bigl[B(X)\bigl(F(X)-L_F(X)\bigl)\Bigr]=\alpha E\bigl[F^{(m)}(W)\bigr]
\end{align}
holds for all $F\in\F^m$.
Then, for an arbitrary function $f\in C_c(\R)$, the class of continuous functions with compact support, consider the function $F:=I^m f:=I^m(f)$ on $\R$. Here, $If(x):=\int_0^x f(t)dt$ and $I^m$ is the $m$-th iterate of $I$. 
Then, $F^{(m)}=f$ and, since $\fnorm{f}<\infty$, it follows from Remark \ref{remmt} (e) that $E\abs{F(X)B(X)}<\infty$ and from \eqref{mainid} and \eqref{uni} we have that 
\begin{align*}
\int_\R f(x)d\mu(x)&=E\bigl[F^{(m)}(V)\bigr]=E\bigl[F^{(m)}(W)\bigr]=\int_\R f(x)d\nu(x)\,. 
\end{align*}
Since the class $C_c(\R)$ is seperating probability measures, this implies that $\mu=\nu$.
\end{proof}

\begin{proof}[Probabilistic existence proof]
From the nonnegativity of $B$ on $J_{m+1}$ we know that 
\begin{equation}\label{casep}
 B(x)\prod_{j=1}^m(x-x_j)\geq0\quad\text{for all }x\in\R\,.
\end{equation}
Thus, if $\alpha\not=0$, it is necessarily positive.
Now, we give the explicit construction of the random variable $\XB$ from \cite{GolRei05b}.\\
Let $Y,U_1,U_2,\dotsc,U_m$ be independent random variables such that $U_j$ has the density $p_j(u):=ju^{j-1}1_{(0,1)}(u)$ ($1\leq j\leq m$) and $Y$ has distribution $\nu$ given by 
\begin{equation}\label{defnu}
 d\nu(y):=\frac{1}{\alpha m!}\prod_{i=1}^m(y-x_i)B(y)d\mu(y)\,,
\end{equation}
where $\mu$ is the distribution of $X$. Note that, by \eqref{casep} and the definition and positivity of $\alpha$,
%$\prod_{i=1}^m(y-x_i)B(y)$ is nonnegative on $\R$ by the assumption on $B$ and that this and the definition of $\alpha$, which is always positive, imply that 
$\nu$ is indeed a probability measure and, hence, such a $Y$ exists. 
Now, we define the random variable 
\begin{equation}\label{XB}
 \XB:=x_m+\sum_{k=1}^m\left(\prod_{i=k}^mU_i\right)(x_{k-1}-x_k)\,,
\end{equation}
where $x_0:=Y$. We claim that $\XB$ satisfies \eqref{mainid}. This claim will be proved by induction on $m=0,1,\dotsc$. If $m=0$, then the claim reduces to 
\[E[B(X)F(X)]=\alpha E[F(Y)]\,,\]
where $\alpha=E[B(X)]$, because $\XB=x_0=Y$ in this case. But the validity of this identity immediately follows from the definition of $\nu=\calL(Y)$ for $m=0$ in \eqref{defnu}. Now, suppose that $m\geq1$ and that the claim is proved for $m-1$. We consider the function $\Btilde(x):=(x-x_m)B(x)$ which has $m-1$ sign changes occuring at the points $x_1<\ldots<x_{m-1}$ and which is such that $\Btilde(x)\geq0$ for $x\geq x_{m-1}$. Furthermore, we have 
\begin{equation*}
 X^{(B)}=(1-U_m)x_m+U_m\left(x_{m-1}+\sum_{k=1}^{m-1}\left(\prod_{i=k}^{m-1}U_i\right)(x_{k-1}-x_k)\right)
\end{equation*}
and, since 
\[\prod_{i=1}^m(y-x_i)B(y)=\prod_{i=1}^{m-1}(y-x_i)\Btilde(y)\]
and 
\[\al m!=E\Bigl[B(X)\prod_{i=1}^{m}(X-x_i)\Bigr]=E\Bigl[\Btilde(X)\prod_{i=1}^{m-1}(X-x_i)\Bigr]\,,\]   
we conclude from the induction hypothesis that %do not have to change $Y=x_0$ in order to know from the induction hypothesis that
\begin{equation*}
\XBtilde:= x_{m-1}+\sum_{k=1}^{m-1}\left(\prod_{i=k}^{m-1}U_i\right)(x_{k-1}-x_k)
\end{equation*}
with the same $Y=x_0$ satisfies the assumptions on a random variable with the generalized $X-\Btilde$ biased distribution. Now, let $F\in\F^m$ be given and define $G:=F'$. Furthermore, for $t\in[0,1]$ we let $G_t(x):=G(x_m+t(x-x_m))$, noting that $G_t^{(k)}(x)=G^{(k)}(x_m+t(x-x_m))t^k$ for $k=0,1,\dotsc,m-1$ and that each $G_t\in\F^{m-1}$. 
Since $\XB=x_m+U_m(\XBtilde-x_m)$ and $\XBtilde, U_m$ are independent, we have from Fubini's theorem that 
\begin{align*}
\alpha E\Bigl[F^{(m)}(\XB)\Bigr]&=\alpha E\Bigl[G^{(m-1)}\bigl(x_m+ U_m(\XBtilde-x_m)\bigr)\Bigr]\\
&=\alpha\int_0^1 E\Bigl[G^{(m-1)}\bigl(x_m+ t(\XBtilde-x_m)\bigr)\Bigr]mt^{m-1}dt\\
&=\alpha m\int_0^1 E\Bigl[G_t^{(m-1)}\bigl(\XBtilde\bigr)\Bigr]dt\,.
\end{align*}
Noting that 
\[\alpha m=\frac{\al m!}{(m-1)!}=\frac{1}{(m-1)!}E\Bigl[\Btilde(X)\prod_{j=1}^{m-1}(X-x_j)\Bigr]\,,\]
we can thus conclude from the induction hypothesis that  
\begin{align}\label{mteq5}
&\alpha E\Bigl[F^{(m)}(\XB)\Bigr]=\int_0^1 E\Bigl[\Btilde(X)\bigl(G_t(X)-L_{G_t}(X)\bigr)\Bigr]dt\notag\\
&=\int_0^1 E\Biggl[\Btilde(X)\biggl(G_t(X)-\sum_{k=1}^{m-1} G_t(x_k)\prod_{\substack{j=1\\j\not=k}}^{m-1}\frac{X-x_j}{x_k-x_j}\biggr)\Biggr]dt\notag\\
&=E\Biggl[B(X)\int_0^1(X-x_m)\biggl(G_t(X)-\sum_{k=1}^{m-1}G_t(x_k)\prod_{\substack{j=1\\j\not=k}}^{m-1}\frac{X-x_j}{x_k-x_j}\biggr)dt\Biggr]\,.
\end{align}

Now, for each real $z\not=x_m$ 
\begin{align}\label{mteq1}
\int_0^1 G_t(z)dt&=\frac{1}{z-x_m}\int_0^1 F'(x_m+t(z-x_m))(z-x_m)dt\notag\\
&=\frac{1}{z-x_m}\int_{x_m}^z F'(s)ds=\frac{F(z)-F(x_m)}{z-x_m}\,, 
\end{align}
implying that 
\begin{align}\label{mteq2}
(x-x_m)\int_0^1 G_t(x)dt=F(x)-F(x_m)
\end{align}
for each $x\in\R$. From \eqref{mteq1} and \eqref{mteq2} we conclude for $x\in\R$ that 
\begin{align*}
 &(x-x_m)\int_0^1\biggl(G_t(x)-\sum_{k=1}^{m-1} G_t(x_k)\prod_{\substack{j=1\\j\not=k}}^{m-1}\frac{x-x_j}{x_k-x_j}\biggr)dt\\
&=F(x)-F(x_m)-\sum_{k=1}^{m-1}\Bigl(F(x_k)-F(x_m)\Bigr)\frac{x-x_m}{x_k-x_m}\prod_{\substack{j=1\\j\not=k}}^{m-1}\frac{x-x_j}{x_k-x_j}\\
&=F(x)-F(x_m)-\sum_{k=1}^{m-1}F(x_k)\prod_{\substack{j=1\\j\not=k}}^{m}\frac{x-x_j}{x_k-x_j}\\
&\quad + F(x_m)\sum_{k=1}^{m-1}\prod_{\substack{j=1\\j\not=k}}^{m}\frac{x-x_j}{x_k-x_j}\,.
\end{align*}
Now, we notice that for each $x\in\R$ we have 
\begin{equation*}
 \sum_{k=1}^{m-1}\prod_{\substack{j=1\\j\not=k}}^{m}\frac{x-x_j}{x_k-x_j} =1-\prod_{j=1}^{m-1}\frac{x-x_j}{x_{m}-x_j}\,,
\end{equation*}
which is clear from the Lagrange form of the interpolation polynomial corresponding to the constant function $1$ and the nodes $x_1,\dotsc,x_m$. Using this, we obtain that 
\begin{align*}
&(x-x_m)\int_0^1\biggl(G_t(x)-\sum_{k=1}^{m-1} G_t(x_k)\prod_{\substack{j=1\\j\not=k}}^{m-1}\frac{x-x_j}{x_k-x_j}\biggr)dt\\
&=F(x)-\sum_{k=1}^mF(x_k)\prod_{\substack{j=1\\j\not=k}}^{m}\frac{x-x_j}{x_k-x_j}\\
&=F(x)-L_F(x)\,.
\end{align*}
Plugging this into \eqref{mteq5} we see that 
\begin{align*}
\alpha E\Bigl[F^{(m)}(\XB)\Bigr]=E\Bigl[B(X)\bigl(F(X)-L_F(X)\bigr)\Bigr]\,,
\end{align*}
as claimed.
\end{proof}

\begin{proof}[Proof of absolute continuity if $m\geq1$]
 To prove this claim, we use the explicit construction of $\XB$ given in \eqref{XB}. Thus, we have that 
\begin{align*}
 \XB&=x_m+\sum_{k=1}^m\left(\prod_{i=k}^mU_i\right)(x_{k-1}-x_k)\\
&=U_1\left(\prod_{i=2}^m U_i\right)\bigl(Y-x_1\bigr)+ \left(x_m+\sum_{k=2}^m\left(\prod_{i=k}^mU_i\right)(x_{k-1}-x_k)\right)\\
&=:U_1 f(Y,U_2,\dotsc,U_m)+g(U_2,\dotsc,U_m)\,.
\end{align*}
Let $N\in\B(\R)$ be a given set such that $\lambda(N)=0$, where $\lambda$ denotes the Lebesgue measure on the line. Then, 
\begin{equation}\label{ac1}
P\bigl(\XB\in N\bigr)=E\Bigl[P\bigl(\XB\in N\,|\,Y,U_2,\dotsc,U_m\bigr)\Bigr]\,.
\end{equation}
Note that $f(Y,U_2,\dotsc,U_m)=\Bigl(\prod_{i=2}^m U_i\Bigr)\bigl(Y-x_1\bigr)\not=0$ $P$- almost surely, since $P(Y=x_j)=0$ for $1\leq j\leq m$ by the definition of $\nu=\calL(Y)$ in \eqref{defnu} and also $P(U_i=0)=0$ for each $i=1,2,\dotsc,m$. Thus, by independence of $U_1$ and $(Y,U_2,\dotsc,U_m)$ we have for each choice of $y\in\R\setminus\{x_1,\dotsc,x_m\}$ and $u_2,\dotsc,u_m\in(0,1)$ that
\begin{align}\label{ac2}
&P\bigl(\XB\in N\,|\,Y=y,U_2=u_2,\dotsc,U_m=u_m\bigr)\notag\\
&=P\bigl(f(y,u_2,\dotsc,u_m)U_1+g(u_2,\dotsc,u_m)\in N\bigr)\notag\\
&=P\Biggl(U_1\in\frac{N-g(u_2,\dotsc,u_m)}{f(y,u_2,\dotsc,u_m)}\Biggr)\notag\\
&=P\bigl(U_1\in\tilde{N}(y,u_2,\dotsc,u_m)\bigr)\,,
\end{align}
where
\[\tilde{N}(y,u_2,\dotsc,u_m)=\frac{N-g(u_2,\dotsc,u_m)}{f(y,u_2,\dotsc,u_m)}
=\Biggl\{\frac{x-g(u_2,\dotsc,u_m)}{f(y,u_2,\dotsc,u_m)}\,\Bigl|\,x\in N\Biggr\}\,.\]
By the properties of the Lebesgue measure it follows that 
\[\lambda\bigl(\tilde{N}(y,u_2,\dotsc,u_m)\bigr)=\frac{1}{\abs{f(y,u_2,\dotsc,u_m)}}\lambda(N)=0\,,\]
so that we conclude that $P(U_1\in\tilde{N}(y,u_2,\dotsc,u_m))=0$, because $U_1$ has an absolutely continuous distribution. Hence, by \eqref{ac2} also
\[P\bigl(\XB\in N\,|\,Y,U_2,\dotsc,U_m\bigr)=0\quad P-\text{almost surely.}\]
Thus, from \eqref{ac1} we infer that 
$P(\XB\in N)=0$. Hence, the distribution of $\XB$ is absolutely continuous with respect to $\lambda$.
\end{proof}

\begin{remark}\label{denformula}
 With the notation of the above existence proof, from the identity 
\begin{equation*}
 E\Bigl[f\bigl(\XB\bigr)\Bigr]=\int_0^1 E\Bigl[f\bigl(x_m+t(\XBtilde-x_m)\bigr)\Bigr]mt^{m-1}dt\,,
\end{equation*}
valid for bounded and measurable $f$, and an easy change of variable one can easily deduce that for $m\geq2$ the ($\lambda$-a.e. unique) density $p$ of $\XB$ is given by 
\begin{equation*}
 p(x)=m\int_0^1\tilde{p}\Bigl(x_m+\frac{x-x_m}{t}\Bigr)t^{m-2}dt\,,
\end{equation*}
where $\tilde{p}$ is the ($\lambda$-a.e. unique) density of $\XBtilde$. This observation may be used to derive density formulas iteratively, beginning with the case $m=1$, see Remark \ref{remmt} (d). It also gives rise to an inductive proof of absolute continuity of the distribution of $\XB$.
\end{remark}

For the zero-bias and the size-bias transformations it is known that if the distribution of the random variable $X$ is a mixture of the distributions of certain variables $X_s$, $s\in S$, then also the biased distribution of $X$ is a mixture of the biased distributions of the $X_s$ (see \cite{Gol10} for the zero-bias case and 
\cite{AGK} for the size-bias case). This property easily generalizes to our situation. We describe it in the abstract situation with a possibly uncountable number of mixed distributions as in \cite{Gol10}. 
Within most applications, though, the set $S$ below will be equal to $\{1,\dotsc,n\}$ for some $n$ and $I$ and $J$ will be  random indices with values in $S$.\\  
Thus, let $(S,\mathcal{S},\gamma)$ be a probability space and let $K:S\times\B(\R)\rightarrow[0,1]$ be a Markov kernel such that with $\mu_s:=K(s,\cdot)$ the distributions $\mu_s$, $s\in S$, satisfy the assumptions of Theorem \ref{maintheo}. 
A random variable $X$ having distribution $\gamma K$ may be constructed as follows. Let $I$ be independent of the family $(X_s)_{s\in S}$, where $I\sim\gamma$ and $X_s$ has distribution $\mu_s$ for each $s\in S$. Then, $X:=X_I$ has distribution $\gamma K$. For each $s\in S$ let 
$\alpha_s:=E[B(X_s)(X_s-x_1)\cdot\ldots\cdot(X_s-x_m)\bigr]$ and assume that 
\begin{align}\label{mix1}
 E\abs{B(X)}&=\int_S E\abs{B(X_s)}d\gamma(s)<\infty\quad\text{and}\notag\\
E\abs{X^mB(X)}&=\int_S E\abs{X_s^mB(X_s)}d\gamma(s)<\infty\,.
\end{align}
From \eqref{mix1} and Remark \ref{remmt} (f) we conclude that
\begin{equation*}
 0<\alpha:=E\bigl[B(X)(X-x_1)\cdot\ldots\cdot(X-x_m)\bigr]=\int_S\alpha_s d\gamma(s)<\infty\,.
\end{equation*}
Further, for each $s\in S$ let $X_s^{(B)}$ have the generalized $X_s-B$ biased distribution. Let $J$ be independent of the family $(X_s^{(B)})_{s\in S}$ having distribution $P(J\in A):=\int_A\frac{\alpha_s}{\alpha}d\gamma(s)$, $A\in\mathcal{S}$. 

\begin{proposition}\label{mix}
 Under the above assumptions the variable $X^{(B)}:=X_J^{(B)}$ has the generalized $X-B$ biased distribution.
\end{proposition}

\begin{proof}
The easy proof is quite standard: For $F\in\F^m$ we have by Fubini's theorem
\begin{align*}
&\;E\bigl[B(X)\bigl(F(X)-L_F(X)\bigr)\bigr]=\int_S E\bigl[B(X_s)\bigl(F(X_s)-L_F(X_s)\bigr)\bigr]d\gamma(s)\\
&=\int_S \alpha_sE\bigl[F^{(m)}\bigl(X_s^{(B)}\bigr)\bigr]d\gamma(s)
=\alpha \int_S \frac{\alpha_s}{\alpha}E\bigl[F^{(m)}\bigl(X_s^{(B)}\bigr)\bigr]d\gamma(s)\\
&=\alpha \int_S E\bigl[F^{(m)}\bigl(X_s^{(B)}\bigl)\bigr] P(J\in ds)
=\alpha E\bigl[F^{(m)}\bigl(X_J^{(B)}\bigr)\bigr]\,. 
\end{align*}
\end{proof}

It is actually not strictly necessary to assume that $X_s$ satisfies the asumptions of Theorem \ref{maintheo} for each $s\in S$. In fact, assuming \eqref{mix1} it follows from Remark \ref{remmt} (f) that $\alpha_s$ exists for $\gamma$-a.e. $s\in S$ but it might be zero for certain values of $s$. Assuming additionally that $\alpha>0$ for $X=X_I$ and letting 
$\XB_s$ have any fixed distribution if $\alpha_s=0$, then the proof goes through as before, since the distribution of the index $J$ puts mass $0$ to values of $s$ such that $\alpha_s=0$.

\subsection{Biasing functions with fewer than $m$ sign changes}\label{s22}

Although Theorem \ref{maintheo} is already quite general, in practice it might happen that one would like the order $m$ of the derivative on the right hand side of \eqref{mainid} to be larger than the number, say $k$, of sign changes of the function $B$ on the left hand side of \eqref{mainid}. 
For example, if $X$ is a nonnegative random variable with finite and non-zero expectation, then $X^e$ is said to have the equilibrium distribution with respect to $X$, if 
\begin{equation}\label{renew}
 E\bigl[f(X)-f(0)\bigr]=E\bigl[X\bigr]E\bigl[f'(X^e)\bigr]
\end{equation}
holds for all Lipschitz-continuous functions $f$. Couplings with this distributional transformation were successfully used for exponential approximation by Stein's\\ method in \cite{PekRol11} and \cite{PekRol11b}.  Thus, it appears as if in \eqref{renew} we would have $m=1$ but $k=0$, since $B=1$. But, as it turns out, 
this distributional transformations is nevertheless covered by Theorem \ref{maintheo} by letting $B(x):=\sign(x)$, for example. Then, as a function on $\R$, $B$ has exactly one sign change at $x_1=0$ and Theorem \ref{maintheo} may be invoked. Since $X$ was assumed nonnegative, this is not quite reflected in equation \eqref{renew}. 
However, there are cases of distributional transformations, which are used in practice and which are not covered by Theorem \ref{maintheo}. For example, in their analysis of the rate of convergence for the distributional convergence of certain random sums of mean zero random variables to the Laplace distribution, in \cite{PiRen} the authors 
use the fact that for each real valued random variable $X$ such that $E[X]=0$ and $0<E[X^2]<\infty$, there exists a unique distribution for a random variable $X^{(L)}$ such that 
\begin{equation}\label{XL}
 E\bigl[f(X)-f(0)\bigr]=\frac{1}{2}E\bigl[X^2\bigr]E\bigl[f''(X^{(L)})\bigr]
\end{equation}
holds for all continuously differentiable functions $f$ with a Lipschitz derivative. In their final version \cite{PiRen} they prove this by giving an explicit construction of the random variable $X^{(L)}$. 
In the first arXiv version, however, they applied Theorem 2.1 of \cite{GolRei05b} with the distributional transformation given by $B(x)=\sign(x)$ twice in a row, and, in order to do so, they had to make sure that the orthogonality assumptions of that theorem were satisfied. This is why they first had to assume that not only $E[X]=0$ but also $P(X<0)=P(X>0)=1/2$ be satisfied. 
Invoking Theorem \ref{maintheo} instead, we are able to prove the following statement, which even generalizes \eqref{XL} to the class of all $X$ with finite second moment. This result is the main building block 
of a generalization of Theorem \ref{maintheo} to cases, where the number of sign changes of $B$ might disagree with the order of the derivative of the test function $F$.

\begin{proposition}\label{mainprop}
 Let $X$ be a real-valued random variable such that $0<E[X^2]<\infty$. Then, for each $a\in\R$, there exists a unique distribution for a random variable $\Xhat_a$ such that 
\begin{equation}\label{Xhat}
 E\bigl[f(X)-f(a)-f'(a)(X-a)\bigr]=\frac{1}{2}E\bigl[(X-a)^2\bigr]E\bigl[f''(\Xhat_a)\bigr]
\end{equation}
 holds for all continuously differentiable functions $f$ with a Lipschitz derivative. Further, the distribution of $\Xhat_a$ is always absolutely continuous with respect to the Lebesgue measure.
\end{proposition}

\begin{remark}\label{mainproprem}
Using the transformation from Proposition \ref{mainprop}, one could easily generalize the results from \cite{PiRen} to random sums with general mean zero summands and even to summands with small, non-zero means. 
\end{remark}

\begin{proof}[Proof of Proposition \ref{mainprop}]
Uniqueness can be seen in a similar way as in the proof of Theorem \ref{maintheo}.
The existence proof is very similar to the proof of Theorem 3.4 in the first arXiv version of \cite{PiRen}: Let $X$ and $f$ be as in the statement of Proposition \ref{mainprop}. Define the function $B$ on $\R$ by 
\begin{equation}\label{sign}
 B(x):=\sign(x-a):=\begin{cases}
                  -1,& x<a\\
                   0,& x=a\\
                   1,& x>a
                 \end{cases}
\end{equation}
having exactly one sign change at $x_1=a$. Thus, since\\ $\alpha:=E[B(X)(X-a)]=E\abs{X-a}\in(0,\infty)$, by Theorem \ref{maintheo}, there exists a random variable $\tilde{X}$ such that 
\begin{equation}\label{mpeq1}
 E\abs{X-a}E\Bigl[g'\bigl(\tilde{X}\bigr)\Bigr]=E\Bigl[\sign(X-a)\bigl(g(X)-g(a)\bigr)\Bigr]
\end{equation}
holds for all Lipschitz functions $g$ on $\R$.   
Now, since for all $x\not=a$
\[\frac{d}{dx}\biggl(\frac{(x-a)^2}{2}\sign(x-a)\biggr)=\abs{x-a}\,,\]
we have 
\begin{align}\label{mpeq2}
 \beta&:=E\Bigl[(\tilde{X}-a)\sign\bigl(\tilde{X}-a\bigr)\Bigr]=E\abs{\tilde{X}-a}=\frac{1}{\alpha}E\Bigl[\sign(X-a)\frac{1}{2}(X-a)^2\sign(X-a)\Bigr]\notag\\
 &=\frac{1}{2\alpha}E\bigl[(X-a)^2\bigr]\in(0,\infty)\,.
\end{align}
Thus, again by Theorem \ref{maintheo}, there exists a random variable $\Xhat_a$ having the $\tilde{X}-B$ biased distribution. This means that 
\begin{align}\label{mpeq3}
 \beta E\bigl[h'(\Xhat_a)\bigr]&= E\Bigl[\sign(\tilde{X}-a)\bigl(h(\tilde{X})-h(a)\bigr)\Bigr]\notag\\
 &=E\Bigl[\sign(\tilde{X}-a)h(\tilde{X})\Bigr]-h(a) E\bigl[\sign(\tilde{X}-a)\bigr]\notag\\
&=E\Bigl[\sign(\tilde{X}-a)h(\tilde{X})\Bigr] -h(a)\frac{1}{\alpha} E[X-a]
\end{align}
holds for all Lipschitz functions $h$.
Since $X$ has finite second moment, one can easily see that \eqref{mpeq1} also holds for absolutely continuous functions $g$ such that $\abs{g'(x)}$ is $O(x)$ as $\abs{x}\to\infty$. In particular this holds for 
$g(x):=\sign(x-a)\bigl(f(x)-f(a)\bigr)$ with $g(a)=0$ and $g'(x)=\sign(x-a)f'(x)$ for $x\not=a$. 
Thus, from \eqref{mpeq1}, \eqref{mpeq2} and \eqref{mpeq3} we conclude that 
\begin{align}\label{mpeq4}
 E\bigl[f(X)-f(a)\bigr]&=E\bigl[\sign(X-a)g(X)\bigr]=\alpha E\bigl[g'(\tilde{X})\bigr]=\alpha E\bigl[\sign(\tilde{X}-a)f'(\tilde{X})\bigr]\notag\\
&=\alpha\beta E\bigl[f''(\Xhat_a)\bigr]+f'(a)E[X-a]\notag\\
&=\frac{1}{2}E\bigl[(X-a)^2\bigr]E\bigl[f''(\Xhat_a)\bigr]+f'(a)E[X-a]\,,
\end{align}
proving \eqref{Xhat}. Absolute continuity of $\mathcal{L}(\Xhat_a)$ follows immediately from Theorem \ref{maintheo}.
\end{proof}

Next, we will use the result of Proposition \ref{mainprop} to give a generalization of Theorem \ref{maintheo} to cases, where the number $k$ of sign changes of $B$ may be smaller than the order $m$ of the derivative we would like to have in the defining identity for the biased distribution.
However, we will have to assume that $k\equiv m \mod 2$, i.e. that $k$ and $m$ have the same parity. 
In what follows, for nonnegative integers $n,j$ we denote by $(n)_j$ the falling factorial, i.e. $(n)_0:=1$ and $(n)_j:=n(n-1)\cdot\ldots\cdot(n-j+1)$ if $j\geq1$.

\begin{theorem}\label{genmaintheo}
 Let $k\leq m$ be nonnegative integers with the same parity and let $B$ be a measurable function on $\R$ having $k$ sign changes at the points $x_1<x_2<\ldots<x_k$ such that $B(x)\geq0$ for all $x\geq x_k$, if $k\geq1$ and for all $x$ in $\R$, if $k=0$. Further, let $X$ be a real-valued random variable such that 
$E\bigl[\abs{B(X) X^j}\bigr]<\infty$ for all $0\leq j\leq m$ and such that 
\begin{equation}\label{alpha2}
 \alpha:= \frac{1}{k!}E\Bigl[B(X)\prod_{j=1}^k\bigl(X-x_j\bigr)\Bigr]\not=0\,.
\end{equation}
If $k=0$, assume further that the generalized $X-B$ biased distribution from Theorem \ref{maintheo} is not the Dirac measure at $0$. 
Then, there exists a unique distribution for a random variable $X^{(B,m)}$ such that 
\begin{equation}\label{mainid2}
 E\Bigl[B(X)\bigl(F(X)-R_F(X)-L_F(X)\bigr)\Bigr]=\beta E\Bigl[F^{(m)}\bigl(\XBm\bigr)\Bigr]
\end{equation}
holds for each $F\in\F^m$, where, with 
\begin{equation}\label{aij}
 a_i^{(j)}:=\sum_{l=1}^k \frac{x_l^{k+j-i-1}}{\prod_{r\not=l}(x_l-x_r)}=\sum_{\substack{(\alpha_1,\dotsc,\alpha_k)\in\N_0^k\\ \sum_{j=1}^k \alpha_j= j-i}} x_1^{\alpha_1}x_2^{\alpha_2}\cdot\ldots\cdot x_k^{\alpha_k}\quad(j\geq i\geq0)\,,  
\end{equation}
we define the polynomial $R_F$ by 
\begin{equation}\label{Rf}
 R_F(x):=R_{F;x_1,\dotsc,x_k}(x):=\prod_{j=1}^k(x-x_j)\sum_{i=0}^{m-k-1}\left(\sum_{j=i}^{m-k-1} \frac{F^{(k+j)}(0)a_i^{(j)}}{(k+j)!}\right) x^i\,,
\end{equation}
if $k\geq1$ and by 
\begin{equation}\label{RF0}
 R_F(x):=\sum_{j=0}^{m-1} \frac{F^{(j)}(0)}{j!}x^j\,,
\end{equation}
if $k=0$. Then, $R_F$ is equal to zero, whenever $k=m$ and has degree at most $m-1$, if $k<m$. Furthermore, $L_F$ still denotes the interpolation polynomial for $F$ corresponding to the nodes $x_1,\ldots,x_k$ given by \eqref{LF} but with $m$ replaced by $k$. Additionally, $\beta$ is always positive and is given by 
\begin{equation}\label{beta}
 \beta:=\frac{1}{m!}E\Bigl[B(X)\Bigl(X^m-\sum_{l=1}^k x_l^m\prod_{r\not=l}\frac{X-x_r}{x_l-x_r}\Bigr)\Bigr]
\end{equation}
if $k\geq1$ and by $\beta=(m!)^{-1}E[B(X)X^m]$, if $k=0$. Also, the distribution of $\XBm$ is always absolutely continuous with respect to the Lebesgue measure unless $k=m=0$. 
\end{theorem}
  
\begin{proof}
 From Theorem \ref{maintheo} we know that $\alpha>0$. Let $F\in\F^m$ be given. By the assumptions on $X$ one can conclude again from Theorem \ref{maintheo} that $E[B(X)(F(X)-L_F(X))]$ exists and that there is a random variable $Y$ having the generalized $X-B$ biased distribution, so that 
\begin{equation}\label{mt2eq1}
 E\Bigl[B(X)\bigl(F(X)-L_F(X)\bigr)\Bigr]=\alpha E\Bigl[F^{(k)}\bigl(Y\bigr)\Bigr]\,.
\end{equation}
From our assumption in the case $k=0$ and from Theorem \ref{maintheo} for $k\geq1$, we know that $Y$ is not almost surely equal to zero. Thus, if $m\geq k+2$, by Proposition \ref{mainprop} (with $a=0$) 
we know that there is a random variable $Y_1$ satisfying 
\begin{align}\label{mt2eq2}
 E\bigl[F^{(k)}(Y)\bigr]&=F^{(k)}(0)+F^{(k+1)}(0)E[Y]\notag\\
&\;+E\bigl[F^{(k)}(Y)-F^{(k)}(0)-F^{(k+1)}(0)Y\bigr]\notag\\
&=F^{(k)}(0)+F^{(k+1)}(0)E[Y] +\beta_1 E\bigl[F^{(k+2)}(Y_1)\bigr]\,,
\end{align}
where $\beta_1=\frac{1}{2}E[Y^2]$. Now, if $m\geq k+4$, then again by Proposition \ref{mainprop} we can find a random variable $Y_2$ such that 
\begin{align}\label{mt2eq3}
 E\bigl[F^{(k+2)}(Y_1)\bigr]&=F^{(k+2)}(0)+F^{(k+3)}(0)E[Y_1]\notag\\
&\;+E\bigl[F^{(k+2)}(Y_1)-F^{(k+2)}(0)-F^{(k+3)}(0)Y_1\bigr]\notag\\
&=F^{(k+2)}(0)+\frac{F^{(k+3)}(0)}{3! \beta_1}E[Y^3]+\beta_2 E\bigl[F^{(k+4)}(Y_2)\bigr]\,,
\end{align}
since $E[Y_1]=\frac{1}{6\beta_1}E[Y^3]=\frac{1}{3!\beta_1}E[Y^3]$ and with 
\[\beta_2=\frac{1}{2}E[Y_1^2]=\frac{1}{2}\frac{1}{12\beta_1}E[Y^4]=\frac{1}{4!\beta_1}E[Y^4]\,.\]
Rearranging \eqref{mt2eq2} and \eqref{mt2eq3} we find that 
\begin{align}\label{mt2eq4}
 E\bigl[F^{(k)}(Y)\bigr]&=F^{(k)}(0)+F^{(k+1)}(0)E[Y]+\frac{F^{(k+2)}(0)}{2}E[Y^2]+\frac{F^{(k+3)}(0)}{3!}E[Y^3]\notag\\
&\;+\frac{1}{4!}E[Y^4] E\bigl[F^{(k+4)}(Y_2)\bigr]\,.
\end{align}
Inductively, for $l=1,\dotsc,\frac{m-k}{2}$ we find that there exists $Y_{l}$ such that, with $Y_0:=Y$ we have 
\begin{align}\label{mt2eq5}
 E\bigl[F^{(k+2l-2)}(Y_{l-1})\bigr]&=F^{(k+2l-2)}(0)+F^{(k+2l-1)}(0)E[Y_{l-1}]\notag\\
&\;+E\bigl[F^{(k+2l-2)}(Y_{l-1})-F^{(k+2l-2)}(0)-F^{(k+2l-1)}(0)Y_{l-1}\bigr]\notag\\
&=F^{(k+2l-2)}(0)+\frac{F^{(k+2l-1)}(0)}{(2l-1)! \beta_{l-1}}E[Y^{2l-1}]+\beta_{l} E\bigl[F^{(k+2l)}(Y_l)\bigr]\,,
\end{align}
where 
\[\beta_l=\frac{1}{(2l)!\beta_1\cdot\ldots\cdot\beta_{l-1}}E\bigl[Y^{2l}\bigr]\,.\]
Again by induction we find the following analog of \eqref{mt2eq4}:
\begin{align}\label{mt2eq6}
 E\bigl[F^{(k)}(Y)\bigr]=\sum_{j=0}^{m-k-1}\frac{F^{(k+j)}(0)}{j!}E\bigl[Y^j\bigr]+\frac{1}{(m-k)!}E\bigl[Y^{m-k}\bigr]E\bigl[F^{(m)}\bigl(Y_{\frac{m-k}{2}}\bigr)\bigr]
\end{align}
Now note that for $j=0,1,\dotsc,m-k$ with the function $F_j(x):=\frac{x^{k+j}}{(k+j)_k}$ we have from \eqref{mt2eq1} that 
\begin{equation}\label{mt2eq7}
E\bigl[Y^j\bigr]=E\bigl[F_j^{(k)}(Y)\bigr]=\frac{1}{\alpha}E\Bigl[B(X)\bigl(F_j(X)-L_{F_j}(X)\bigr)\Bigr]\,.
\end{equation}
Clearly, $Q_j(x):=F_j(x)-L_{F_j}(x)$ is a polynomial of degree $k+j$ having the zeroes $x_1<\ldots<x_k$. Thus, there exists a polynomial $q_j$ of degree $j$ such that $Q_j(x)=q_j(x)\prod_{l=1}^k(x-x_l)$.
Now, first suppose that $k=0$. Then, we have $F_j(x)=Q_j(x)=q_j(x)=x^j$. Thus, from \eqref{mt2eq6} and \eqref{mt2eq7} we can conclude that 
\begin{align}\label{mt2eqa}
 \alpha E\bigl[F(Y)\bigr]&=\sum_{j=0}^{m-1}\frac{F^{(j)}(0)}{j!} E\bigl[B(X)X^j\bigr]\notag\\
&\; +\frac{1}{m!} E\bigl[B(X)X^m\bigr]E\bigl[F^{(m)}\bigl(Y_{\frac{m}{2}}\bigr)\bigr]\,.
\end{align}
Letting $\XBm:=Y_{\frac{m}{2}}$ the claim follows in the case $k=0$ from \eqref{mt2eq1} and \eqref{mt2eqa}. From now on, we will assume that $k\geq1$.
In order to find $q_j$ in this case, we write 
\begin{align}\label{mt2eq8}
 x^{k+j}-L_{x^{k+j}}&=x^{k+j}-\sum_{l=1}^kx_l^{k+j}\prod_{r\not=l}\frac{x-x_r}{x_l-x_r}
=\sum_{l=1}^k\bigl(x^{k+j}-x_l^{k+j}\bigr)\prod_{r\not=l}\frac{x-x_r}{x_l-x_r}\notag\\
&=\sum_{l=1}^k(x-x_l)\sum_{i=0}^{j+k-1}x^ix_l^{j+k-1-i}\prod_{r\not=l}\frac{x-x_r}{x_l-x_r}\notag\\
&=\prod_{r=1}^k(x-x_r)\sum_{i=0}^{j+k-1}x^i\sum_{l=1}^k\frac{x_l^{k+j-1-i}}{\prod_{r\not=l}(x_l-x_r)}\notag\\
&=\prod_{r=1}^k(x-x_r)\sum_{i=0}^{j}x^i\sum_{l=1}^k\frac{x_l^{k+j-1-i}}{\prod_{r\not=l}(x_l-x_r)}\,,
\end{align}
the last identity because the left hand side is a polynomial of degree $j+k$ and, hence, the right hand side must also be. Thus, as a neat by-product we have proved that 
\begin{equation}\label{byproduct}
 \sum_{l=1}^k\frac{x_l^n}{\prod_{r\not=l}(x_l-x_r)}=0\quad\text{ for all }n\leq k-2\,.
\end{equation}
From \eqref{mt2eq8} we conclude that $q_j$ is given by 
\begin{equation}\label{mt2eq9}
 q_j(x)=\frac{1}{(k+j)_k}\sum_{i=0}^{j}\left(\sum_{l=1}^k\frac{x_l^{k+j-1-i}}{\prod_{r\not=l}(x_l-x_r)}\right)x^i=\frac{1}{(k+j)_k}\sum_{i=0}^{j} a_i^{(j)}x^i\,.
\end{equation}
Hence, from \eqref{mt2eq7} and \eqref{mt2eq9} we find for $j=0,1,\dotsc,m-k$ that 
\begin{equation}\label{mt2eq10}
 E\bigl[Y^j\bigr]=\frac{1}{\alpha(k+j)_k}\sum_{i=0}^{j} a_i^{(j)}E\Bigl[B(X)X^i\prod_{l=1}^k(X-x_l)\Bigr]
\end{equation}
Plugging this into \eqref{mt2eq6} we arrive at 
\begin{align}\label{mt2eq11}
 E\bigl[F^{(k)}(Y)\bigr]&=\sum_{j=0}^{m-k-1}\frac{F^{(k+j)}(0)}{j!}\frac{1}{\alpha(k+j)_k}\sum_{i=0}^{j} a_i^{(j)}E\Bigl[B(X)X^i\prod_{l=1}^k(X-x_l)\Bigr] \notag\\
&\; +\frac{1}{\alpha(m-k)!(m)_k}\sum_{i=0}^{m-k} a_i^{(m-k)}E\Bigl[B(X)X^i\prod_{l=1}^k(X-x_l)\Bigr] E\bigl[F^{(m)}\bigl(Y_{\frac{m-k}{2}}\bigr)\bigr]\notag\\
&=\frac{1}{\alpha}\sum_{i=0}^{m-k-1} E\Bigl[B(X)X^i\prod_{l=1}^k(X-x_l)\Bigr]\sum_{j=i}^{m-k-1}\frac{F^{(k+j)}(0)a_i^{(j)}}{(k+j)!}\notag\\
&\;+\frac{1}{\alpha m!}\sum_{i=0}^{m-k} a_i^{(m-k)}E\Bigl[B(X)X^i\prod_{l=1}^k(X-x_l)\Bigr] E\bigl[F^{(m)}\bigl(Y_{\frac{m-k}{2}}\bigr)\bigr]\notag\\
&=\frac{1}{\alpha}E\bigl[B(X)R_F(X)\bigr]\notag\\
&\;+\frac{1}{\alpha m!}\sum_{i=0}^{m-k} a_i^{(m-k)}E\Bigl[B(X)X^i\prod_{l=1}^k(X-x_l)\Bigr] E\bigl[F^{(m)}\bigl(Y_{\frac{m-k}{2}}\bigr)\bigr]
%+\frac{\beta}{\alpha}E\bigl[F^{(m)}\bigl(Y_{\frac{m-k}{2}}\bigr)\bigr]\,.
\end{align}
Now, from reading \eqref{mt2eq8} backwards (with $m=k+j$) we obtain  
\begin{align}\label{mt2eq12}
 \sum_{i=0}^{m-k}a_i^{(m-k)}\prod_{j=1}^k(x-x_j)x^i&=\sum_{i=0}^{m-k}\sum_{l=1}^k\frac{x_l^{m-1-i}}{\prod_{r\not=l}(x_l-x_r)}\prod_{j=1}^k(x-x_j)x^i\notag\\
&=x^m-L_{x^m}=x^m-\sum_{l=1}^k x_l^m\prod_{r\not=l}\frac{x-x_r}{x_l-x_r}\,.
\end{align}
Thus, from \eqref{mt2eq11} and \eqref{mt2eq12} we see that 
\begin{equation}\label{mt2eq13}
  E\bigl[F^{(k)}(Y)\bigr]=\frac{1}{\alpha}E\bigl[B(X)R_F(X)\bigr]+\frac{\beta}{\alpha}E\bigl[F^{(m)}\bigl(Y_{\frac{m-k}{2}}\bigr)\bigr]\,.
\end{equation}
Letting  $\XBm:=Y_{\frac{m-k}{2}}$ \eqref{mainid2} now follows from \eqref{mt2eq1} and \eqref{mt2eq13}.\\ 
To see that $\beta>0$, note that we know from our assumption in the case $k=0$ and from Theorem \ref{maintheo}  in the case $k\geq1$ that $Y$ cannot almost surely be equal to zero. Thus, the even moments of $Y$ are also non-zero.
Since we know from \eqref{mt2eq6} that $\beta=\frac{\alpha}{(m-k)!}E[Y^{m-k}]$ with $\alpha>0$ and as $m-k$ is even, it follows that also $\beta>0$. Knowing that $\beta$ is necessarily positive, uniqueness of the distribution for $\XBm$ can be proved as for $\XB$ in the proof of Theorem \ref{maintheo}. Absolute continuity of 
$\mathcal{L}(\XBm)$ in the case that not both, $m$ and $k$ are equal to zero, now follows from Theorem \ref{maintheo} and Proposition \ref{mainprop}. It remains to show the alternative representation for the numbers $a_i^{(j)}$ in \eqref{aij}. This is given by Lemma \ref{aijlemma}.
\end{proof}

\begin{lemma}\label{aijlemma}
For $k\geq1$ let $x_1,\dotsc,x_k$ be distinct real (or complex) numbers. Then, for each nonnegative integer $n$ we have the identity 
\begin{equation*}
 \sum_{l=1}^k\frac{x_l^n}{\prod_{r\not=l}(x_l-x_r)}=\sum_{\substack{(\alpha_1,\dotsc,\alpha_k)\in\N_0^{k}\\ \sum_{j=1}^k \alpha_j= n-k+1}} x_1^{\alpha_1}x_2^{\alpha_2}\cdot\ldots\cdot x_k^{\alpha_k}\,.
\end{equation*}
\end{lemma}

\begin{proof}
We prove the claim by induction on $k$, simultaneously for all $n\geq0$. If $k=1$, then it is clearly true. Now assume that $k\geq1$ and that $x_1,\dotsc,x_k,x_{k+1}$ are distinct numbers. Then, we can write
\begin{align*}\label{aij1}
 \sum_{l=1}^{k+1}\frac{x_l^n}{\prod_{r\not=l}(x_l-x_r)}&=\sum_{l=1}^{k}\frac{x_l^n-x_{k+1}^n}{\prod_{r\not=l}(x_l-x_r)}
+x_{k+1}^n\sum_{l=1}^{k+1}\frac{1}{\prod_{r\not=l}(x_l-x_r)}\notag\\
&=:S_1+S_2\,.
\end{align*}
Noting that 
\begin{equation*}
x_l^n-x_{k+1}^n = \bigl(x_l-x_{k+1}\bigr)\sum_{i=0}^{n-1}x_l^ix_{k+1}^{n-1-i}\,,
\end{equation*}
we conclude from the induction hypothesis that 
\begin{align*}\label{aij2}
S_1&=\sum _{l=1}^{k}\frac{\bigl(x_l-x_{k+1}\bigr)}{\prod_{\substack{r=1\\ r\not=l}}^{k+1}(x_l-x_r)}\sum_{i=0}^{n-1}x_l^ix_{k+1}^{n-1-i}
=\sum_{i=0}^{n-1} x_{k+1}^{n-1-i}\sum _{l=1}^{k}\frac{x_l^i}{\prod_{\substack{r=1\\ r\not=l}}^{k}(x_l-x_r)}\notag\\
&=\sum_{i=0}^{n-1} x_{k+1}^{n-1-i}\sum_{\substack{(\beta_1,\dotsc,\beta_k)\in\N_0^{k}\\ \sum_{j=1}^k \beta_j= i-k+1}} x_1^{\beta_1}x_2^{\beta_2}\cdot\ldots\cdot x_k^{\beta_k}\\
&=\sum_{\substack{(\beta_1,\dotsc,\beta_k,\beta_{k+1})\in\N_0^{k+1}\\ \sum_{j=1}^k \beta_j= n-(k+1)+1}} x_1^{\beta_1}x_2^{\beta_2}\cdot\ldots\cdot x_k^{\beta_k}x_{k+1}^{\beta_{k+1}}\,.
\end{align*}
Thus, it only remains to show that $S_2=0$. But this follows from \eqref{byproduct}, completing the proof.  
\end{proof}

\begin{remark}\label{grem}
\begin{enumerate}[(a)]
\item We may call the distribution of $X^{(B,m)}$ the \textit{$X-(B,m)$ biased distribution}. Note, however, that, as for $\XB$, the distribution of $X^{(B,m)}$ is sensitive to the number $k$ and the choice of the sign change points\\ $x_1<\ldots<x_k$, if these are ambiguous (see Remark \ref{remmt} (b)).
\item It is easy to see that an analog of Proposition \ref{mix} also exists for the $X-(B,m)$ biased distribution. 
\item As in Proposition \ref{mainprop}, we could introduce $\frac{m-k}{2}$ additional location parameters $a_j\in\R$ in the statement of Theorem \ref{genmaintheo}. This can be seen from the proof, which invokes 
Proposition \ref{mainprop} exactly $\frac{m-k}{2}$ times with $a=0$. We have, however, decided to refrain from this in order to keep the result more readable and, because it is not clear, which would be the most useful choice of the $a_j$ for typical applications (see Theorem \ref{highorder}). It should be clear, however, how the proof and the statement would have to be modified, if one wanted to introduce such extra parameters. 
\item One can see from examples that the condition that $k$ and $m$ have the same parity cannot be abandoned without substitution. In fact, if $X$ has support equal to $\R$, then one cannot find a random variable $X^e$ such that \eqref{renew} is satisfied for all Lipschitz $f$, because it is easy to see that 
the corresponding distribution would need to have a density proportional to $q(t):=1_{[0,\infty)}(t)P(X>t)-1_{(-\infty,0)}(t)P(X\leq t)$, which is negative for $t<0$. Note that contrarily, if there is an $x_1\in\R$ such that $X\geq x_1$ almost surely (and $E[X]>x_1$), then letting $B(x):=\sign(x-x_1)$ having one sign change at $x_1$, by Theorem \ref{maintheo} 
we find a random variable $X^{(B)}$ such that $E[f(X)-f(x_1)]=\alpha E[f'(X^{(B)})]$ with $\alpha=E[X-x_1]$. 
\item In view of (d) it would be nice to know, if, for each real random variable $X$ with $E\abs{X}<\infty$, we can find another random variable $Y$ and constants $\beta>0$ and $c_f$, $f$ Lipschitz on $\R$, such that 
\begin{equation}\label{cp1}
E\bigl[f(X)-c_f\bigr]=\beta E\bigl[f'(Y)\bigr] 
\end{equation}
holds for each Lipschitz function $f$. By Remark \ref{grem} (c) this is true for all $X$, which are almost surely bounded below. Thus, only those $X$ with support equal to $\R$ must be considered to find a counterexample. Note that such a counterexample would imply that the condition that $k$ and $m$ in Theorem \ref{genmaintheo} 
have the same parity is also necessary, in general.
\end{enumerate}
\end{remark}

\section{Examples and Applications}\label{examples}

\subsection{First order Stein operators}\label{fo}
In this Subsection we give some examples of first-order distributional transformations, whose existence is guaranteed by Theorem \ref{maintheo} and demonstrate how this theory may be applied to prove certain Stein type characterizations without using the solution of the corresponding Stein equation. We also show, how one can use a coupling of $X$ and $\XB$ to estimate the distance of $\calL(X)$ to a fixed point of the distributional transformation 
induced by $B$. Finally, we show by examle that the distribution of $X^{(B)}$ in general depends on the choice of 
the zeroes of $B$, if these are ambiguous.

\begin{example}
 \begin{enumerate}[(a)]
  \item Let $X$ be a real-valued random variable with $0<E[X^2]<\infty$. Choosing $B(x)=x$ with a single sign change at $0$, we conclude from Theorem \ref{maintheo} that there exists a random variable $X^{gz}$ such that 
  \begin{equation}\label{genzero}
E\bigl[X\bigl(f(X)-f(0)\bigr)\bigr]=E\bigl[X^2\bigr]E\bigl[f'(X^{gz})\bigr]
\end{equation}
holds for all Lipschitz-continuous functions $f$ on $\R$. Obviously, if $X$ has mean zero, then $X^{gz}$ has the $X$-zero biased distribution from \cite{GolRei97}. Thus, in general, we say that $X^{gz}$
has the \textit{generalized $X$-zero biased distribution} and we call the mapping $\calL(X)\mapsto\calL(X^{(gz)})$ the \textit{generalized zero bias transformation}.
\item Under the same assumptions on $X$ as in (a) we now choose $B(x):=x-E[X]$. Then, 
\begin{equation*}
 \alpha=E\bigl[(X-E[X])^2\bigr]=\Var(X)
 \end{equation*}
and, again by Theorem \ref{maintheo}, we find that there is a random variable $X^{nz}$ such that 
\begin{equation}\label{nonzero}
 E\bigl[(X-E[X])f(X)\bigr]%=E\bigl[\bigl(X-E[X]\bigr)\bigl(f(X)-f(E[X])\bigr)\bigr]
 =\Var(X)E\bigl[f'(X^{nz})\bigr]\,,
\end{equation}
where we have used that $E[B(X)]=0$ in this case. Again, whenever $X$ has mean zero, the distribution of $X^{nz}$ reduces to the $X$-zero biased distribution. In general, we call it the 
\textit{$X$-non-zero biased distribution}. Note that the existence of this distribution already follows from Theorem 2.1 in \cite{GolRei05b}, as $B$ satisfies their orthogonality relation in this case. 
\end{enumerate}
\end{example}

Next, we show by example how the existence of such distributional transformations may be used to prove a Stein type characterization of a given distribution, which is a fixed point of the 
distributional transformation. We first need the following definition.
\begin{definition}
 Let $\sigma>0$ and $Z_\sigma\sim N(0,\sigma^2)$. Then, the distribution of $Y_\sigma:=\abs{Z_\sigma}$ is called the \textit{half-normal distribution} or \textit{modulus normal distribution} with parameter $\sigma^2=E[Y_\sigma^2]$. 
 Further, we say that $W_\sigma$ has the \textit{negative half-normal distribution} with parameter $\sigma^2$, if $-W_\sigma$ has the half-normal distribution with parameter $\sigma^2$.
\end{definition}

\begin{proposition}\label{charid}
Let $X$ be a real-valued random variable such that $0<E[X^2]<\infty$. Then $\calL(X)$ is a fixed point of the generalized zero bias transformation if and only if it is a mixture of a half-normal and a negative half-normal distribution with the same parameter.    
\end{proposition}

\begin{proof}
Let the distribution of $X$ be a fixed point of the generalized zero-bias transformation. Then, from Remark \ref{remmt} (d) we know that $X$ has an absolutely continuous distribution with density $p$ given by 
\begin{equation}\label{pform}
p(t)=\frac{1}{E[X^2]}E\Bigl[X\bigl(1_{\{0\leq t\leq X\}}-1_{\{X<t<0\}}\bigr)\Bigr]\,.
\end{equation}
For $t>0$ we thus have 
\begin{align}\label{pform1}
p(t)&=\frac{1}{E[X^2]}E\Bigl[X1_{\{0\leq t\leq X\}}\Bigr]=\frac{1}{E[X^2]}\int_\R s1_{[t,\infty)}(s) p(s)ds\notag\\
&=\frac{1}{E[X^2]}\int_t^\infty sp(s)ds\,.
\end{align}
Similarly, for $t<0$ we can show that 
\begin{equation}\label{pform2}
p(t)=\frac{-1}{E[X^2]}\int_{-\infty}^tsp(s)ds\,.
\end{equation}
From \eqref{pform1} and \eqref{pform2} we conclude that $p$ is continuously differentiable on $(0,\infty)$ and on 
$(-\infty,0)$ and that 
\begin{equation}\label{dglp}
p'(t)=\frac{-tp(t)}{E[X^2]}
\end{equation}
for each $t\not=0$. From \eqref{dglp} we see, that 
\begin{equation}\label{prep1}
p(t)=p(0+)\exp\Bigl(\frac{-t^2}{2E[X^2]}\Bigr)
\end{equation}
for $t>0$ and 
\begin{equation}\label{prep2}
p(t)=p(0-)\exp\Bigl(\frac{-t^2}{2E[X^2]}\Bigr)
\end{equation}
for $t<0$. Here, we used the shorthands $p(0+):=\lim_{t \downarrow 0}p(t)$ and\\ $p(0-):=\lim_{t\uparrow 0}p(t)$.
The claim now follows from \eqref{prep1} and \eqref{prep2}.\\
Conversely, if the distribution of $X$ is such a mixture, then, by a standard computation involving Fubini's theorem, one easily verifies that $X$ satisfies 
\[E\bigl[X^2\bigr]E\bigl[f'(X)\bigr]=E\bigl[X\bigl(f(X)-f(0)\bigr)\bigr]\,,\]
and, hence, that $\calL(X)$ is a fixed point of the generalized zero bias transformation. We omit the details.
\end{proof}

From Proposition \ref{charid} we directly infer the following Stein characterization of the class of half-normal distributions, whose derivation does not make use of the solution to any Stein equation.

\begin{corollary}
A nonnegative random variable $X$ with $0<E[X^2]<\infty$ has the half-normal distribution with parameter $\sigma^2=E[X^2]$, if and only if 
\begin{equation}\label{Steinidhn}
E\bigl[X^2\bigr]E\bigl[f'(X)\bigr]=E\bigl[X\bigl(f(X)-f(0)\bigr)\bigr]
\end{equation}
holds for all Lipschitz-continuous functions $f:[0,\infty)\rightarrow\R$.
\end{corollary}

\begin{remark}\label{charrem}
\begin{enumerate}[(a)]
 \item The statement of Proposition \ref{charid} can be generalized to more general biasing functions $B$ with one sign change point $x_1$ such that $B(x)(x-x_1)\geq0$ on $\R$. Indeed, in this case, one can derive the formula 
\begin{align*}
 p'(t)=-\frac{1}{\alpha}B(t)p(t)
\end{align*}
for all $t\not=x_1$, which is analogous to \eqref{dglp} and which implies that the $\log$-derivative of $p$ is given by $-B/\alpha$. Hence, the family of denisties $p$ giving rise to fixed points of the distributional transformation 
$\calL(X)\mapsto\calL(X^{(B)})$ can be reconstructed as before. 
\item Suppose that the distribution of $Z$ is a fixed point of the distributional transformation in (a). Up to dividing $B$ by a constant, which does not change the distributional transformation, we can assume that 
\[E\bigl[B(Z)(Z-x_1)\bigr]=1\,,\]
i.e. $-B$ is the $\log$-derivative of the density $p$ of $Z$. 
Then, the Stein equation from the density approach (see e.g. \cite{CGS}) for $Z$ corresponding to a test function $h$ such that $E[h(Z)]$ exists, reads
\begin{equation*}
 f'(x)-B(x)f(x)=h(x)-E[h(Z)]
\end{equation*}
and is solved by 
\begin{equation*}
 f_h(x)=\frac{1}{p(x)}\int_a^b\bigl(h(t)-E[h(Z)]\bigr)p(t)dt\,,
\end{equation*}
where we suppose that the support of $\calL(Z)$ is given by the interval $\abquer$ for some $-\infty\leq a<b\leq\infty$.
The law of $Z$ is  then usually chracterized by the identity 
\begin{equation}\label{chardens}
 E\bigl[f'(Z)-B(Z)f(Z)\bigr]=f(b-)-f(a+)\,,
\end{equation}
valid for all functions $f$ from some large function class $\F$. 
If $h$ is Lipschitz-continuous, one typically has bounds for $f_h$ of the form 
\begin{equation*}
 \fnorm{f_h}\leq c_0\fnorm{h'}\,,\quad\fnorm{f_h'}\leq c_1\fnorm{h'}\quad\text{and}\quad\fnorm{f_h''}\leq c_2\fnorm{h'}
\end{equation*}
for some finite constants $c_0,c_1$ and $c_2$ (see \cite{CGS}, again).\\
Now, suppose that $X$ is given and that $\XB$ has the generalized $X-B$ biased distribution and is constructed on the same space as $X$. Then, for a $1$-Lipschitz function $h$, we can estimate 
\begin{align}
 &\;\Babs{E\bigl[h(X)\bigr]-E\bigl[h(Z)\bigr]}=\Babs{E\bigl[f_h'(X)-B(X)f_h(X)\bigr]}\notag\\
 &=\Babs{E\bigl[f_h'(X)-f_h'(\XB)\bigr]+E\bigl[f_h'(\XB)\bigr](1-\alpha)-f_h(x_1)E\bigl[B(X)\bigr]}\notag\\
 &\leq c_2E\babs{X-\XB}+c_1\abs{1-\alpha}+\babs{f_h(x_1)}\babs{E\bigl[B(X)\bigr]}\label{da1}\\
 &\leq c_2E\babs{X-\XB}+c_1\abs{1-\alpha}+c_0\babs{E\bigl[B(X)\bigr]}\label{da2}\,,
 \end{align}
where 
\begin{equation*}
 \alpha=E\bigl[B(X)(X-x_1)\bigr]\,.
\end{equation*}
From \eqref{chardens} with $f(x)=x-x_1$ we presume that $\alpha$ should be close to one, if $\calL(X)\approx\calL(Z)$. Thus, the second term in \eqref{da2} (or \eqref{da1}) 
should be close to zero. Also, if we can couple $\XB$ close to $X$, then the first term should be small, too. In many cases, we have that $E[B(Z)]=0$, as is suggested by taking $f(x)\equiv1$ in \eqref{chardens}, 
and from which we conclude 
that the third term in \eqref{da2} is also close to zero and, hence that \eqref{da2} gives a good estimate of the \textit{Wasserstein distance}
\begin{equation*}
 d_\W\bigl(\calL(X),\calL(Z)\bigr)=\sup_{h\in\Lip(1)}\Babs{E\bigl[h(X)\bigr]-E\bigl[h(Z)\bigr]}
\end{equation*}
between $\calL(X)$ and $\calL(Z)$. Here, $\Lip(1)$ denotes the class of $1$-Lipschitz functions $h$. However, there are examples where $E[B(Z)]\not=0$ and, hence, where 
one cannot expect $\eqref{da2}$ to be small. For instance, if $Z$ has the exponential distribution with parameter $1$, then $\frac{d}{dx}\log p(x)\equiv-1$ and the function
$B(x):=\sign(x)$ on $\R$ has one sign change at $0$ and satisfies $E[B(Z)]=1$. Furtunately, in this case one can show that 
\[f_h(0+)=\lim_{x\downarrow0}f_h(x)=0\]
and, hence, \eqref{da1} might still give a useful estimate.\\
In a nutshell, if the distribution of $Z$ is a fixed point of the distributional transformation induced by $B$ and we somehow conjecture that $\calL(X)\approx\calL(Z)$ and 
if we can can couple $X$ and $\XB$ sufficiently close, then we should be able to accurately estimate the (Wasserstein) distance between $\calL(X)$ and $\calL(Z)$ by the above procedure.  
\end{enumerate}
\end{remark}

The following example illustrates the dependence of the distribution of $X^{(B)}$ on the choice of the sign change points, if there are non-trivial intervals, where $B$ vanishes identically and, if the 
orthogonality relations from \cite{GolRei05b} do not hold.
\begin{example}\label{ambi}
 Let $m=1$ and consider a measurable function $B:\R\rightarrow\R$ such that there are real numbers $a<b$ with $B(x)\leq0$ for $x\in(-\infty,a]$, $B(x)=0$ for $x\in(a,b]$ and $B(x)\geq0$ for $x\in(b,\infty)$. Also, 
 let $X$ be a real-valued random variable such that $E\abs{X^jB(X)}<\infty$ for $j=0,1$ and suppose that 
 \begin{equation*}
  \alpha:=E\bigl[B(X)(X-a)\bigr]\not=0\quad\text{and}\quad\beta:=E\bigl[B(X)(X-b)\bigr]\not=0\,.
 \end{equation*}
From Remark \ref{remmt} (d) we know that a density $p$ for the distribution of $X^{(B;a)}$ is given by 
\begin{equation*}
 p(t)=\frac{1}{\alpha}E\bigl[B(X)\bigl(1_{\{a\leq t\leq X\}}-1_{\{X< t<a\}}\bigr)\bigr]
\end{equation*}
and that a density $q$ for the distribution of $X^{(B;b)}$ is given by 
\begin{equation*}
 q(t)=\frac{1}{\beta}E\bigl[B(X)\bigl(1_{\{b\leq t\leq X\}}-1_{\{X< t<b\}}\bigr)\bigr]\,.
\end{equation*}
A trite computation then shows that 
\begin{align}\label{amb1}
 \alpha p(t)&=1_{[a,\infty)}(t)E\bigl[B(X)1_{\{X\geq t\}}\bigr]-1_{(-\infty,a)}(t)E\bigl[B(X)1_{\{X< t\}}\bigr]\notag\\
 &=\beta q(t)+1_{[a,b)}(t)\Bigl(E\bigl[B(X)1_{\{X< t\}}\bigr]+E\bigl[B(X)1_{\{X\geq t\}}\bigr]\Bigr)\notag\\
 &=\beta q(t)+E[B(X)]1_{[a,b)}(t)\,.
 \end{align}
We immediately see that, if the orthogonality relation $E[B(X)]=0$ is satisfied, then $\alpha=\beta$ and $p=q$. This is in accordance with the fact that under this condition the $X^{(B)}$ distribution 
is the same for all choices of the zero point of $B$ as stated in \cite{GolRei05b}. If, however, $E[B(X)]\not=0$, then we see from \eqref{amb1} that $p$ and $q$ are generally different and, hence, that the distribution of 
$X^{(B)}$ actually depends on the choice of the zeroes of $B$.\\
For a concrete example, let $X$ be uniformly distributed on $[-1,1]$ and let $B(x)=\max(x,0)=:x^+$. Then, with the notation of the situation above, we can let 
$a=-1$ and $b=0$ and obtain $E[B(X)]=E[X^+]=1/4$ as well as
\begin{equation*}
 \alpha=E\bigl[X^+(X+1)\bigr]=\frac{5}{12}\quad\text{and}\quad\beta=E\bigl[XX^+\bigr]=E\bigl[X^21_{\{X\geq0\}}\bigr]=\frac{1}{6}\,.
\end{equation*}
Hence, in this case 
\begin{align*}
 q(t)&=6E\bigl[X^+1_{\{X\geq t\geq0\}}\bigr]=6\frac{1}{2}\int_t^1s ds\,1_{[0,1]}(t)=\frac{3}{2}(1-t^2)1_{[0,1]}(t)\quad\text{and}\\
 p(t)&=\frac{12}{5}\cdot\frac{1}{6}\cdot\frac{3}{2}(1-t^2)1_{[0,1]}(t)+\frac{1}{4}\cdot\frac{12}{5}1_{[-1,0)}(t)\\
 &=\frac{3}{5}(1-t^2)1_{[0,1]}(t)+\frac{3}{5}1_{[-1,0)}(t)\,.
\end{align*}
Obviously, $p$ and $q$ give rise to two different distributions.
\end{example}

\subsection{Higher order Stein operators}\label{ho}
The purpose of this Subsection is to show, how the existence of certain couplings guaranteed by Theorem \ref{genmaintheo} can be used to assess the distance of the distribution 
of a given random variable $X$ to the distribution of a random variable $Z$, which is characterized by some higher order linear Stein operator $L$ of the form 
\begin{equation}\label{steinop}
 Lf(x)=f^{(m)}(x)-\sum_{j=0}^{m-1}B_j(x)f^{(j)}(x)\,,
\end{equation}
where $m$ is a nonnegative integer and $B_j:\R\rightarrow\R$ is a Borel-measurable function, $j=0,1,\dotsc,m-1$.\\
We first consider the special case $m=2$ of second order Stein operators. We do so for two reasons: 
Firstly, it may be instructive to first consider the easiest particular case that goes beyond the class of first order operators. Secondly, and more importantly, in the case $m=2$ we benefit from the fact that we allowed for an additional parameter $a\in\R$ in Proposition \ref{mainprop}, whereas we refrained from introducing such parameters in the general Theorem \ref{genmaintheo}. In this case, the operator $L$ becomes
\begin{equation*}
 Lf(x)=f''(x)-B_1(x)f'(x)-B_0(x)f(x)\,.
\end{equation*}

\begin{proposition}\label{secor}
Suppose that the functions $B_0,B_1:\R\rightarrow\R$ are Borel-measurable and that there is an $a\in\R$ such that $B_1(x)(x-a)\geq0$ for all $x\in\R$ and that $B_0$ is nonnegative on $\R$.
 Furthermore assume that we are given a real-valued random variable $X$ such that the expressions
 \[E\abs{B_0(X)}\,,\quad E\abs{B_1(X)}\,,\quad E\abs{X^2B_0(X)} \quad\text{and}\quad E\abs{XB_1(X)}\]
 are all finite and such that $\alpha:=\alpha_1+\alpha_2>0$, where
 \begin{equation*}
 \alpha_1=\frac{1}{2}E\bigl[B_0(X)(X-a)^2\bigr]\quad\text{and}\quad\alpha_2=E[B_1(X)(X-a)]\,.
\end{equation*}
Then, there exists a unique distribution for a random variable $X^*$ such that for all $f\in\F^2$ we have 
\begin{equation*}
 \alpha E\bigl[f''(X^*)\bigr]=E\bigl[B_0(X)\bigl(f(X)-f(a)-f'(a)(X-a)\bigr)+B_1(X)\bigl(f'(X)-f'(a)\bigr)\bigr]\,.
\end{equation*}
The law of $X^*$ is always absolutely continuous with respect to the Lebesgue measure.
\end{proposition}

\begin{proof}
Uniqueness is proved in the same way as in the proof of Theorem \ref{maintheo}. So let us just prove the existence of $X^*$. 
First, choose $\Xhat_a$ as in Proposition \ref{mainprop}. Let $\nu:=\calL(\Xhat_a)$ and, if $E[B_0(\Xhat_a)]\not=0$, define $\mu$ by 
 \begin{equation*}
  d\mu(x):=\frac{B_0(x)}{E[B_0(\Xhat_a)]}d_\nu(x)\,,
 \end{equation*}
whereas, if $E[B_0(\Xhat_a)]=0$, let $\mu:=\delta_0$. Finally, let $Y_1\sim\mu$ and construct $Y_2$ and a random index $I\in\{1,2\}$ on the same probability space as $Y_1$ such that 
$I$ is independent of $Y_1,Y_2$, and $Y_2$ has the generalized $X-B_1$-biased distribution and 
\begin{equation*}
 P(I=j)=\frac{\alpha_j}{\alpha_1+\alpha_2}\,,\quad j=1,2\,.
 \end{equation*}
 Note that this implies that 
 \begin{align*}
  \alpha_2E\bigl[g'(Y_2)\bigr]&=E\bigl[B_1(X)\bigl(g(X)-g(a)\bigr)\bigr]\quad\text{and}\\
  \alpha_1E\bigl[f''(Y_1)\bigr]&=E\bigl[B_0(X)\bigl(f(X)-f(a)-f'(a)(X-a)\bigr)\bigr]
 \end{align*}
hold for all sufficiently smooth functions $g$ and $f$, respectively. Hence, letting $X^*:=Y_I$ we have for all $f\in\F^2$ that 
\begin{align*}
 \alpha E\bigl[f''(X^*)\bigr]&=\sum_{j=1}^2\alpha_j E\bigl[f''(Y_j)\bigr]\\
 &=E\bigl[B_0(X)\bigl(f(X)-f(a)-f'(a)(X-a)\bigr)+B_1(X)\bigl(f'(X)-f'(a)\bigr)\bigr]\,,
\end{align*}
as claimed. Also, note that the distribution of $X^*$, being a mixture of absolutely continuous distributions, is itself absolutely continuous.
\end{proof}

\begin{remark}\label{remso}
\begin{enumerate}[(a)]
\item One can check that the function $q$ with 
\begin{equation*}
 q(t):=\frac{1}{\alpha}E\Bigl[\bigl(B_1(X)+B_0(X)(X-t)\bigr)\bigl(1_{\{a\leq t\leq X\}}-1_{\{X<t<a\}}\bigr)\Bigr]
\end{equation*}
is the ($\lambda$-a.e. unique) probability density function of $X^*$.
\item If the operator $L$ in \eqref{steinop} with $m=2$ is characterizing for the distribution of $Z$, then the Stein equation corresponding to a test function $h$ with $E\abs{Z}<\infty$ is given by 
\begin{equation*}
 Lf(x)=h(x)-E[h(Z)]\,,
\end{equation*}
and, often, it has a solution $f_h$ such that the lower order derivatives can be uniformly bounded by constants, i.e. $\fnorm{f_h^{(i)}}\leq c_i$ uniformly over $h$ in some class $\mathcal{H}$ of 
test functions. Then, if one can couple the given random variable $X$ to a $X^*$ such as in Proposition \ref{secor}, then one can easily show that 
\begin{align}\label{coupbound}
 \babs{E[h(X)]-E[h(Z)]}&=\babs{E\bigl[Lf_h(X)\bigr]}\leq c_3E\babs{X-X^*}+\babs{1-\alpha}c_2\\
 &\;+\abs{f'(a)}\babs{E\bigl[B_0(X)(X-a)+B_1(X)\bigr]}+\babs{f(a)}\babs{E\bigl[B_0(X)\bigr]}\notag\,.
 \end{align}
Now, in typical cases one either has that the quantities $f_h(a)$ and $f_h'(a)$ are equal to zero (as is the case for the operator used in \cite{PRR13}), or the expressions 
\[\babs{E\bigl[B_0(X)(X-a)+B_1(X)\bigr]}\quad\text{and}\quad\babs{E\bigl[B_0(X)\bigr]}\]
are close to zero. The latter could be guessed from choosing $f(x)=x-a$ and $f(x)=1$, respectively, together with the assumption that $\calL(X)\approx \calL(Z)$. The same heuristic 
applied to $f(x)=\frac{1}{2}(x-a)^2$ suggests that $\alpha$ should be close to $1$.
Thus, the right hand side of \eqref{coupbound} should be close to zero, if $X$ and $X^*$ are coupled close to each other.
\item As in the first-order case (see Remark \ref{charrem}) one can show that if $\calL(Z)$ is a fixed point of the distributional transformation from Proposition \ref{secor}, then its density $p$ satisfies the 
second order linear differential equation
\begin{equation*}
 \alpha p''(t)=\bigl(B_0(t)-B_1'(t)\bigr)p(t)-B_1(t)p'(t)\,,
 \end{equation*}
from which one should be able to reconstruct the class of fixed points in practice by exploiting boundary conditions like $\int p(t)dt=1$.
\end{enumerate}
\end{remark}

Now, we return to the case of a general $m\geq1$. Henceforth, we denote by $R_{j,f}$ and $L_{j,f}$, respectively, the polynomials from the statement of Theorem \ref{genmaintheo} for $B=B_j$, $j=0,1,\dotsc,m-1$ and define  $Q_{j,f}:=L_{j,f}+R_{j,f}$. In Theorem \ref{highorder} below, we make the assumption that $B_j$ has $0\leq k_j\leq m-j$ sign changes and that $k_j\equiv m-j \mod 2$. Then, by Theorem \ref{genmaintheo}, $Q_{j,f}$ is a polynomial of degree $\leq m-j-1$, $j=0,1,\dotsc,m-1$. Also, assume that $X$ is a real random variable such that 
$E\abs{B_j(X)X^l}<\infty$ for each $0\leq l\leq m-j$ and $0\leq j\leq m-1$. Then, for $j=0,1,\dotsc,m-1$, we define 
\begin{equation}\label{betaj}
 \beta_j:=\frac{1}{(m-j)!}E\Bigl[B_j(X)\Bigl(X^{m-j}-\sum_{l=1}^{k_j} x_l^{m-j}\prod_{r\not=l}\frac{X-x_r}{x_l-x_r}\Bigr)\Bigr]\,,
\end{equation}
which is always nonnegative by Theorem \ref{genmaintheo}. %The following consequence of Theorem \ref{genmaintheo} is the main observation towards this goal.

\begin{theorem}\label{highorder}
 With the above notation and assumptions, suppose that for each $j=0,1,\dotsc,m-1$ the function $B_j$ has $0\leq k_j\leq m-j$ sign changes, where $k_j\equiv m-j \mod 2$. 
 Furthermore, assume that there is some $j\in\{0,1,\dotsc,m-1\}$ such that $\beta_j>0$ and let $\beta:=\sum_{j=0}^{m-1}\beta_j>0$. Then, there exists a unique distribution for a 
 random variable $X^*$ such that for all $f\in\F^m$ we have 
 \begin{equation}\label{hoid}
  \sum_{j=0}^{m-1}E\Bigl[B_j(X)\bigl(f^{(j)}(X)-Q_{j,f^{(j)}}(X)\bigr)\Bigr]=\beta E\Bigl[f^{(m)}\bigl(X^*\bigr)\Bigr]\,.
 \end{equation}
The law of $X^*$ is always absolutely continuous with respect to the Lebesgue measure.
 \end{theorem}

\begin{proof}
Again, we only prove the existence part.
For each $j=0,1,\dotsc,m-1$ let $Y_j$ have the $X-(B_j,m-j)$ biased distribution, whenever $\beta_j\not=0$ and let $Y_j:=0$, otherwise. Also, let $I\in\{0,1,\dotsc,m-1\}$ be a random index, which is independent 
 of $Y_0,Y_1,\dotsc,Y_{m-1}$ such that 
 \begin{equation*}
  P(I=j)=\frac{\beta_j}{\sum_{l=0}^{m-1}\beta_l}\,,\quad j=0,1,\dotsc, m-1 
 \end{equation*}
and define $X^*:=Y_I$. Then, with the notation $f_j:=f^{(j)}$, $j=0,1,\dotsc,m-1$, by Theorem \ref{genmaintheo} we have 
\begin{align*}
 \beta E\Bigl[f^{(m)}\bigl(X^*\bigr)\Bigr]&=\sum_{j=0}^{m-1}\beta_jE\Bigl[f^{(m)}\bigl(Y_j\bigr)\Bigr]=\sum_{j=0}^{m-1}\beta_jE\Bigl[f_j^{(m-j)}\bigl(Y_j\bigr)\Bigr]\notag\\
 &=\sum_{j=0}^{m-1}E\Bigl[B_j(X)\bigl(f_j(X)-R_{j,f_j}(X)-L_{j,f_j}(X)\bigr)\Bigr]\\
 &=\sum_{j=0}^{m-1}E\Bigl[B_j(X)\bigl(f^{(j)}(X)-Q_{j,f^{(j)}}(X)\bigr)\Bigr]\,.
\end{align*}
\end{proof}

\begin{remark}
 It is possible that a coupling of $X$ and $X^*$ as in Theorem \ref{highorder} will be useful to bound the distance of the distribution of $X$ to that of $Z$ also in the case $m>3$, once 
 such Stein operators are used in practice. Maybe it would first be necessary to adjust this distributional transformation slightly by introducing additional location parameters $a_j$ related 
 to the functions $B_j$, as discussed in Remark \ref{grem} (c).
\end{remark}

%The next result addresses the particular case of second order Stein-operators. It is not quite a corollary of Theorem \ref{highorder} because it exploits the fact that we can choose 
%another free location parameter $a$ in Proposition \ref{mainprop}. It would, however, become a direct consequence of Theorem \ref{highorder}, if we translated the involved functions and 
%distributions by $a$.

\section{Analytical proof of Theorem \ref{maintheo}}\label{anproof}

\begin{lemma}\label{hilfslemma1}
 Let $F:\R\rightarrow\R$ be $m$-times differentiable for an integer $m\geq0$ such that $f:=F^{(m)}>0$ on $\R$. Then, $F$ has at most $m$ zeroes.
\end{lemma}

\begin{proof}
We prove the claim by induction on $m$. Since $F^{(0)}=f>0$ has no zeroes if $m=0$, the assertion is clear in this case. Now, let $m\geq1$ and assume that the claim is true for $(m-1)$-times differentiable functions. Suppose, contrarily, that $F$ has $m+1$ distinct zeroes $y_1<\ldots<y_m<y_{m+1}$. Then, by Rolle's theorem there exist points 
$z_k\in(y_k,y_{k+1})$ such that $F'(z_k)=0$ for $k=1,\dotsc,m$. Since the points $z_1,\dotsc,z_m$ are necessarily pairwise distinct zeroes of the $(m-1)$-times differentiable function $G:=F'$ with $G^{(m-1)}=F^{(m)}>0$, 
this contradicts the induction hypothesis.
\end{proof}

\begin{lemma}\label{hilfslemma2}
Let $m\geq0$ be an integer and let $G\in C^m(\R)$ be a function such that $f(x):=G^{(m)}(x)\geq\epsilon$ for all $x\in\R$, where $\eps>0$. Then, for each fixed real number $a$ and all $x\geq a$ we have that
\begin{equation*}
 G(x)\geq\sum_{j=0}^{m-1} G^{(j)}(a)\frac{(x-a)^j}{j!}+\frac{\eps}{m!}(x-a)^m\,.
\end{equation*}

Hence, for each polynomial $Q$ of degree at most $m-1$ it follows that\\ $\lim_{x\to\infty}(G(x)+Q(x))=+\infty$ if $m\geq1$ and that $\liminf_{x\to\infty}G(x)\geq\eps$ if $m=0$.
\end{lemma}

\begin{proof}
 The second assertion follows easily from the first one. We prove the first claim by induction on $m\geq0$. If $m=0$, then $G^{(m)}(x)=f(x)\geq\eps$ for each $x\in\R$, which is the claim for $m=0$. Now, assume that $m\geq1$ and that the claim holds for $m-1$. Then, from the fundamental theorem of calculus and the induction hypothesis we 
conclude that for all $x\geq a$

\begin{align*}
 G(x)&=G(a)+\int_a^x G'(t)dt\\
&\geq G(a)+\int_a^x\left(\sum_{j=0}^{m-2}G^{(j+1)}(a)\frac{(t-a)^j}{j!}+\frac{\eps}{(m-1)!}(t-a)^{m-1}\right)dt\\
&=G(a)+\sum_{j=0}^{m-2}G^{(j+1)}(a)\frac{(x-a)^{j+1}}{(j+1)!}+\frac{\eps}{m!}(x-a)^m\\
&=\sum_{j=0}^{m-1} G^{(j)}(a)\frac{(x-a)^j}{j!}+\frac{\eps}{m!}(x-a)^m\,.
\end{align*}
\end{proof}

Recall that for real numbers $x_1<x_2<\ldots<x_m$ we let $J_1:=(-\infty,x_1]$, $J_k:=(x_{k-1},x_k]$ for $2\leq k\leq m$ and $J_{m+1}:=(x_m,\infty)$.

\begin{lemma}\label{hauptlemma1}
Let $f:\R\rightarrow\R$ be a nonnegative continuous function and let $x_1<x_2<\ldots<x_m$ be real numbers. Then, there is a unique function $F\in C^m(\R)$ such that $F^{(m)}=f$ and $(-1)^{m+1-k}F(x)\geq0$ for all $x\in J_k$ and each $1\leq k\leq m+1$. 
\end{lemma}

\begin{proof}
 We first prove the easier uniqueness claim. Let $F_1$ and $F_2$ be two such functions. Since $F_1^{(m)}-F_2^{(m)}=f-f=0$ identically, we know that $Q:=F_1-F_2$ is a polynomial with degree at most $m-1$. By continuity we have $Q(x_k)=F_1(x_k)-F_2(x_k)=0-0=0$ for each $k=1,2,\dotsc,m$. This implies that $Q$ must be the zero polynomial, i.e. 
$F_1=F_2$.\\
Now we turn to the existence of $F$. We first assume that there is an $\eps>0$ such that $f(x)\geq\eps$ for each $x\in\R$. We define the function $G:=I_{x_m}^m f$, where, for a real number $a$, we let $I_af(x):=(I_af)(x):=\int_a^x f(t)dt$ and $I_a^m$ is the $m$-th iterate of the operator $I_a$. Then, $G\in C^m(\R)$ and $G^{(m)}=f$. Furthermore, we have $G(x_m)=0$ and one 
can easily see by induction on $m$ that $G(x)>0$ for all $x\in J_{m+1}$. Since, in general, $G(x_k)\not=0$ for $1\leq k\leq m-1$, we let $L:=L_{G;x_1,\ldots,x_m}$ be the unique (interpolation) polynomial of degree $\leq m-1$ such that $L(x_k)=G(x_k)$ for $1\leq k\leq m$ and define $F:=G-L$. Of course, it holds that $F^{(m)}=G^{(m)}=f$. 
Further, by construction we have $F(x_k)=G(x_k)-L(x_k)=0$ for all $k=1,\dotsc,m$. By Lemma \ref{hilfslemma1}, we conclude that $F$ has exactly the zeroes $x_1,\dotsc,x_m$. In particular, either $F(x)>0$ for each $x\in J_{m+1}$ or $F(x)<0$ for each $x\in J_{m+1}$. The second alternative being impossible by Lemma \ref{hilfslemma2} and the intermediate 
value theorem we conclude that $F$ is strictly positive on $J_{m+1}$. Next, we make sure that $F$ really changes signs at the points $x_1,\dotsc,x_m$. Since $F(x_k)=0$ it is enough to show that $F'(x_k)\not=0$ for each $k=1,\dotsc,m$. From $F(x_k)=F(x_{k+1})=0$ and Rolle's theorem we know that there exist $z_k\in(x_{k},x_{k+1})$ 
such that $F'(z_k)=0$ for $1\leq k\leq m-1$. Since, again by Lemma \ref{hilfslemma1}, we know that $F'$ has at most $m-1$ zeroes, it follows that $F'(x_k)\not=0$ for each $k=1,\dotsc,m$. Thus, $F$ also satisfies the second condition from the statement of the lemma.\\
Now, we only asume that $f\geq0$ is nonnegative and for each $n\in\N$ we let 
$f_n:=f+1/n$. Then, for each $n\geq1$, $f_n\geq n^{-1}$ satisfies the assumptions of the case just treated. Additionally, the sequence $(f_n)_{n\in\N}$ converges uniformly to $f$. This implies that $I_af_n$ converges to $I_a f$ uniformly on compact intervals (for each $a\in\R$), yielding that also $G_n:=I_{x_m}^mf_n$ converges to $G:=I_{x_m}^mf$ 
uniformly on compacts. By the specific Lagrange form of the interpolation polynomial, one can 
easily see that also $L_{G_n;x_1,\ldots,x_m}$ converges pointwise to $L_{G;x_1,\ldots,x_m}$ as $n\to\infty$. Thus, letting $F_n:=G_n-L_{G_n;x_1,\ldots,x_m}$ and $F:=G-L_{G;x_1,\ldots,x_m}$ we know from the first case that  $(-1)^{m+1-k}F_n(x)\geq0$ for all $x\in J_k$ and each $1\leq k\leq m+1$ and since 
$F_n(x)\stackrel{n\to\infty}{\longrightarrow}F(x)$ for each $x\in\R$, the same applies to $F$.
\end{proof}

\begin{lemma}\label{hauptlemma2}
Let $B$ be a measurable biasing function on $\R$, having $m\in\N$ sign changes occuring at the points $x_1<\ldots<x_m$ as above. Then, for each nonnegative, continuous function $f$ on $\R$, there exists a unique function $F$ on $\R$ such that $F^{(m)}=f$ and $F(x)\cdot B(x)\geq0$ for all $x\in\R$. 
Furthermore, letting $G_f:=G_{f;x_1,\dotsc,x_m}:=I_{x_m}^m f$ and denoting by $L_{G_f}$ the interpolation polynomial of degree at most $m-1$ corresponding to the function $G_f$ and to the nodes $x_1,\dotsc,x_m$, we have that $F=G_f-L_{G_f}$.
%letting $G_f(x):=G_{f;x_1,\dotsc,x_m}(x):=I_{x_m}^m f(x)$ and denoting by $Q_f(x):=Q_{f;x_1,\ldots,x_m}$ the interpolation polynomial of degree at most $m-1$ corresponding to the function $-G_f$ and to the nodes such that $Q_f(x_k)=-G_f(x_k)$ for $1\leq k\leq m$, we have that $F=G_f+Q_f$.
\end{lemma}

\begin{proof}
This follows immediately from Lemma \ref{hauptlemma1} and its proof.
\end{proof}

\begin{proof}[Analytical proof of Theorem \ref{maintheo}]
From the first lines of the probabilistic existence proof, which are independent of the remainder of that proof, we already know that $\alpha=(m!)^{-1}E[B(X)(X-x_1)\cdot\ldots\cdot(X-x_m)]>0$. Further, we concentrate on the non-trivial case that $m\geq1$.
We define the operator $T:C_c(\R)\rightarrow\R$ by 
\begin{equation}\label{defT}
 Tf:=\frac{1}{\alpha m!}E\Bigl[B(X)\bigl(G_f(X)-L_{G_f}(X)\bigr)\Bigr]\,,
\end{equation}
with $G_f$ and $L_{G_f}$ as in the statement of Lemma \ref{hauptlemma2}. Since $\fnorm{f}<\infty$ for $f\in C_c(\R)$, $T$ is well-defined by the assumptions on $X$ and $B$. It is also easy to see that $T$ is linear. In order to invoke the Riesz representation theorem, we aim at showing that $T$ is also positive. Thus, let
$f\in C_c(\R)$ be nonnegative. By Lemma \ref{hauptlemma2} we know that $F(x)\cdot B(x)\geq0$ for all $x\in\R$, where $F(x)=G_f(x)-L_{G_f}(x)$. This immediately implies that $Tf\geq0$ and, hence, $T$ is a positive, linear operator on $C_c(\R)$. By the Riesz representation theorem there exists a unique (positive) Radon measure 
$\nu$ on $(\R,\B(\R))$ such that 
\begin{equation}\label{aneq1}
Tf=\int_\R f(x)d\nu(x)\quad\text{for each }f\in C_c(\R)\,.
\end{equation}
In order to show that $\nu$ is in fact a probability measure, we choose nonnegative functions $f_n\in C_c(\R)$, $n\geq1$, such that $f_n\nearrow1$ pointwise. Since the functions $f_n$ are uniformly bounded (by $1$), one can show similarly as in the proof of Lemma \ref{hauptlemma1}, that the $G_{f_n}$ converge to $G_1$ pointwise 
as $n\to\infty$ and one can show inductively that $\abs{G_{f_n}}\leq \abs{G_1}$, where $G_1(x)=\frac{(x-x_m)^m}{m!}$. Thus, since $B(X)G_1(X)$ is integrable by the assumptions of Theorem \ref{maintheo}, we conclude from the dominated convergence theorem that 
\begin{align}\label{aneq2}
\lim_{n\to\infty}Tf_n&=\lim_{n\to\infty}\frac{1}{\alpha m!}E\Bigl[B(X)\bigl(G_{f_n}(X)-L_{G_{f_n}}(X)\bigr)\Bigr]\notag\\
&=\frac{1}{\alpha m!}E\Bigl[B(X)\bigl(G_1(X)-L_{G_1}(X)\bigr)\Bigr]\,.
\end{align}
Note that by construction $Q(x):=G_1(x)-L_{G_1}(x)$ is a polynomial of degree $m$ such that $Q(x_k)=0$ for $k=1,\dotsc,m$. Hence, there exists $c\not=0$ such that $Q(x)=c\prod_{k=1}^m(x-x_k)$. Since $c=Q^{(m)}=G_1^{(m)}=1$, we conclude from \eqref{aneq2} that
\begin{equation}\label{aneq3}
 \lim_{n\to\infty}Tf_n=\frac{1}{\alpha m!}E\Bigl[B(X)\prod_{k=1}^m(X-x_k)\Bigr]=\frac{\alpha m!}{\alpha m!}=1\,.
\end{equation}
On the other hand, by the monotone convergence theorem and \eqref{aneq1} we have 
\begin{equation}\label{aneq4}
\lim_{n\to\infty}Tf_n =\lim_{n\to\infty}\int_\R f_n(x)d\nu(x)=\int_\R1 d\nu=\nu(\R)\,.
\end{equation}
From \eqref{aneq3} and \eqref{aneq4} we conclude that $\nu$ is indeed a probability measure.\\
Thus, we can choose a random variable $\XB$ on some probability space with distribution $\nu$. In order to show that $\XB$ satisfies \eqref{mainid}, we let $F\in\F^m$ be given. Then, since $F^{(m-1)}$ is Lipschitz, we know that $f:=F^{(m)}$ exists almost everywhere and is bounded. Let $(f_n)_{n\geq1}$ be a sequence 
in $C_c(\R)$ converging to $f$ pointwise such that $\fnorm{f_n}\leq\fnorm{f}$ for all $n\geq1$. Such a sequence can be constructed by convolution with suitable mollifiers with compact support, for example. Then, by an argument similar to that leading to \eqref{aneq2}, one can see, using \eqref{defT}, \eqref{aneq1} and the dominated convergence theorem twice, that 
\begin{align}\label{aneq5}
 \alpha E\Bigl[F^{(m)}\bigl(\XB\bigr)\Bigr]&=\alpha E\Bigl[f\bigl(\XB\bigr)\Bigr]=\alpha\lim_{n\to\infty}E\Bigl[f_n\bigl(\XB\bigr)\Bigr]\notag\\
&=\lim_{n\to\infty}E\Bigl[B(X)\bigl(G_{f_n}(X)-L_{G_{f_n}}(X)\bigr)\Bigr]\notag\\
&=E\Bigl[B(X)\bigl(G_f(X)-L_{G_f}(X)\bigr)\Bigr]\,.
\end{align}
Now, it is easily seen by successive differentiation that $F=G_f+T_{m-1,x_m}F$, where $T_{m-1,x_m}F$ is the Taylor polynomial of order $m-1$ around $x_m$ corresponding to $F$. Since the interpolation polynomial of degree $\leq m-1$ corresponding to $T_{m-1,x_m}F$ is still $T_{m-1,x_m}F$, this implies that 
\begin{equation*}
 L_F=L_{G_f+T_{m-1,x_m}F}=L_{G_f}+L_{T_{m-1,x_m}F}=L_{G_f}+T_{m-1,x_m}F
\end{equation*}
and, hence, 
\begin{equation}\label{aneq6}
 F-L_F=G_f+T_{m-1,x_m}F-\bigl(L_{G_f}+T_{m-1,x_m}F\bigr)=G_f-L_{G_f}\,.
\end{equation}
From \eqref{aneq5} and \eqref{aneq6} it finally folllows that 
\begin{equation*}
\alpha E\Bigl[F^{(m)}\bigl(\XB\bigr)\Bigr]=E\Bigl[B(X)\bigl(F(X)-L_F(X)\bigr)\Bigr]\,,
\end{equation*}
which was to be proved.
\end{proof}

\begin{acknowledgements}
Major parts of this work have been carried out while I was postdoc at TU M\"unchen, Germany. 
I would like to thank Professor Gesine Reinert for inviting me to a visit to Oxford in September 2013 and giving me the opportunity of presenting parts of this work during my stay there.
I am also grateful to an anonymous referee whose comments helped me improve the presentation and exposition of the above results.
\end{acknowledgements}

\normalem
\bibliographystyle{alpha}      % mathematics and physical sciences
\bibliography{doebler}   % name your BibTeX data base

\begin{thebibliography}{AGK13}

\bibitem[AG10]{ArrGol}
R.~Arratia and L.~Goldstein.
\newblock {Size bias, sampling, the waiting time paradox, and inifinite
  divisibility: when is the increment independent?}
\newblock {\em {\tt arXiv:1007.3910}}, 2010.

\bibitem[AGK13]{AGK}
R.~Arratia, L.~Goldstein, and F.~Kochman.
\newblock {Size bias for one and all}.
\newblock {\em {\tt arXiv:1308.2729}}, 2013.

\bibitem[BC05]{BarCh}
A.~D. Barbour and Louis H.~Y. Chen, editors.
\newblock {\em An introduction to {S}tein's method}, volume~4 of {\em Lecture
  Notes Series. Institute for Mathematical Sciences. National University of
  Singapore}.
\newblock Singapore University Press, Singapore, 2005.
\newblock Lectures from the Meeting on Stein's Method and Applications:a
  Program in Honor of Charles Stein held at the National University of
  Singapore, Singapore, July 28--August 31, 2003.

\bibitem[CGS11]{CGS}
Louis H.~Y. Chen, Larry Goldstein, and Qi-Man Shao.
\newblock {\em Normal approximation by {S}tein's method}.
\newblock Probability and its Applications (New York). Springer, Heidelberg,
  2011.

\bibitem[Gol10]{Gol10}
Larry Goldstein.
\newblock Bounds on the constant in the mean central limit theorem.
\newblock {\em Ann. Probab.}, 38(4):1672--1689, 2010.

\bibitem[GR96]{GolRin96}
Larry Goldstein and Yosef Rinott.
\newblock Multivariate normal approximations by {S}tein's method and size bias
  couplings.
\newblock {\em J. Appl. Probab.}, 33(1):1--17, 1996.

\bibitem[GR97]{GolRei97}
Larry Goldstein and Gesine Reinert.
\newblock Stein's method and the zero bias transformation with application to
  simple random sampling.
\newblock {\em Ann. Appl. Probab.}, 7(4):935--952, 1997.

\bibitem[GR05]{GolRei05b}
Larry Goldstein and Gesine Reinert.
\newblock Distributional transformations, orthogonal polynomials, and {S}tein
  characterizations.
\newblock {\em J. Theoret. Probab.}, 18(1):237--260, 2005.

\bibitem[PR11a]{PekRol11b}
Erol Pek{\"o}z and Adrian R{\"o}llin.
\newblock Exponential approximation for the nearly critical {G}alton-{W}atson
  process and occupation times of {M}arkov chains.
\newblock {\em Electron. J. Probab.}, 16:no. 51, 1381--1393, 2011.

\bibitem[PR11b]{PekRol11}
Erol~A. Pek{\"o}z and Adrian R{\"o}llin.
\newblock New rates for exponential approximation and the theorems of {R}\'enyi
  and {Y}aglom.
\newblock {\em Ann. Probab.}, 39(2):587--608, 2011.

\bibitem[PR14]{PiRen}
John Pike and Haining Ren.
\newblock Stein's method and the {L}aplace distribution.
\newblock {\em ALEA Lat. Am. J. Probab. Math. Stat.}, 11(1):571--587, 2014.

\bibitem[PRR13]{PRR13}
Erol~A. Pek{\"o}z, Adrian R{\"o}llin, and Nathan Ross.
\newblock Degree asymptotics with rates for preferential attachment random
  graphs.
\newblock {\em Ann. Appl. Probab.}, 23(3):1188--1218, 2013.

\end{thebibliography}

 \end{document}